\documentclass[11pt,letterpaper]{article}

\usepackage[margin=0.8in]{geometry}
\usepackage[latin1]{inputenc}
\usepackage{amsmath, amssymb, amsbsy, mathdots, mathabx}
\usepackage{mathtools}
\usepackage{mathrsfs}
\usepackage{amsthm}
\usepackage{enumitem,makeidx}
\usepackage[all]{xy}
\usepackage{tikz}
\usepackage{changepage, comment}
\usepackage[mathcal]{euscript}
\usepackage{thmtools, thm-restate}
\usepackage{url}

\title{On the cohomology of integral $p$-adic unipotent radicals}
\author{Niccol\`o Ronchetti}
\date{\today}

\theoremstyle{theorem}
\newtheorem{thm}{Theorem}
\newtheorem{lem}[thm]{Lemma}
\newtheorem{prop}[thm]{Proposition}
\newtheorem{fact}[thm]{Fact}
\newtheorem{claim}[thm]{Claim}
\newtheorem{cor}[thm]{Corollary}

\theoremstyle{definition}
\newtheorem{defn}{Definition}

\theoremstyle{remark}
\newtheorem*{rem}{Remark}

\newcommand{\Hom}{\ensuremath{\mathrm{Hom}}}
\newcommand{\Ext}{\ensuremath{\mathrm{Ext}}}
\newcommand{\Tor}{\ensuremath{\mathrm{Tor}}}
\newcommand{\Aut}{\ensuremath{\mathrm{Aut}}}
\newcommand{\Mod}{\ensuremath{\mathrm{Mod}}}
\newcommand{\Lie}{\ensuremath{\mathrm{Lie} \,}}
\newcommand{\Span}{\ensuremath{\mathrm{span}}}

\newcommand{\Gal}{\ensuremath{\mathrm{Gal}}}

\newcommand\floor[1]{\lfloor#1\rfloor}
\newcommand\ceil[1]{\lceil#1\rceil}

\newcommand{\id}{\ensuremath{\mathrm{id}}}
\newcommand{\gr}{\ensuremath{\mathrm{gr}}} 

\newcommand{\ad}{\ensuremath{\mathrm{ad}}}
\newcommand{\im}{\ensuremath{\mathrm{Im}}}

\newcommand{\lra}{\ensuremath{\longrightarrow}}

\newcommand{\Qp}{\ensuremath{\mathbb Q_p}}

\newcommand{\N}{\ensuremath{\mathbb N}}

\newcommand{\FFp}{\ensuremath{\overline { \mathbb F_p}}}

\newcommand{\Z}{\ensuremath{\mathbb Z}}
\newcommand{\Zp}{\ensuremath{\mathbb Z_p}}

\newcommand{\C}{\ensuremath{\mathbb C}}

\newcommand{\Res}[3]{\ensuremath{\mathrm{Res}_{#1 / #2}  #3 }} 
\newcommand{\OO}{\ensuremath{\mathcal O}}

\begin{document}
\maketitle
\begin{abstract}
    Let $\mathrm G$ be a reductive split $p$-adic group and let $\mathrm U$ be the unipotent radical of a Borel subgroup. We study the cohomology with trivial $\Zp$-coefficients of the profinite nilpotent group $N = \mathrm U(\OO_F)$ and its Lie algebra $\mathfrak n$, by extending a classical result of Kostant to our integral $p$-adic setup. The techniques used are a combination of results from group theory, algebraic groups and homological algebra.
\end{abstract}
\section{Introduction}
In this paper we study the cohomology of the group of integral points of unipotent radicals of $p$-adic algebraic groups.

Let $p \ge 5$ be a prime number. Let $F$ be a $p$-adic field, with ring of integer $\OO_F$, and denote $d = [F : \Qp]$.
Let $\mathrm G$ be a connected, reductive, split $F$-group. Fix once and for all a maximal split $F$-torus $\mathrm T$, as well as a Borel subgroup $\mathrm B$ containing $\mathrm T$ and defined over $F$. Let $\mathrm U = \mathcal R_u(\mathrm B)$ be the unipotent radical of $\mathrm B$.

We also fix integral, smooth $\OO_F$-models for all the group schemes considered above (see \ref{notationsection} for their construction), and we still denote them by $\mathrm G, \mathrm B, \mathrm U, \mathrm T$.

Our main object of study is the continuous cohomology of the profinite group $\mathrm U(\OO_F)$ with coefficients into finitely generated $\Zp$-modules, in particular the trivial module, and we want to understand these objects as $\mathrm T(\OO_F)$-modules.

Our first result is a slight modification of a celebrated theorem of Kostant that describes Lie algebra cohomology for the unipotent $\OO_F$-Lie algebra  $\mathfrak u = \Lie \mathrm U$. We prove a `$p$-adic integral' version, very similar (but not identical) to those already proven in \cite{PT,vigre,GK}.
\begin{restatable}{thm}{kostant}[Kostant's theorem] \label{kostantthm} Let $\mathrm G$ be simple and simply connected.
Suppose that $p \ge h$, the Coxeter number of $\mathrm G$.
Consider the trivial $\mathfrak u$-module $\OO_F$.
Then its cohomology is \[ H^n_{\OO_F} \left( \mathfrak u, \OO_F \right) \cong \bigoplus_{w \in W, l(w) = n} V_{\OO_F}(w \cdot 0) \] where $V_{\OO_F}(\lambda)$ is the highest-weight $\mathrm T$-module associated\footnote{in other words, $V_{\OO_F} (\lambda) \cong \OO_F$ with $\mathrm T(\OO_F)$ acting via $\lambda$.} to the character $\lambda$ and $l(w)$ denotes the length of the Weyl element $w$, defined as in section \ref{notationsection}.
\end{restatable}
This is proved in section \ref{kostantsubsection}, mainly following the ideas in \cite{vigre}.

Let $\mathfrak n$ be the Lie algebra $\mathfrak u$, viewed as a $\Zp$-algebra (see section \ref{notationsection} for the precise setup).
Since we are interested in cohomology with $\Zp$-coefficients, we record the following corollary which follows from theorem \ref{kostantthm}.
\begin{restatable}{cor}{kostantcorollary} \label{kostantcor} Suppose $F / \Qp$ is Galois.

The only $\Res{\OO_F}{\Zp} \mathrm T (\Zp) = \mathrm T(\OO_F)$-modules appearing in the cohomology $ H^n_{\Zp} \left( \mathfrak n, \Zp \right)$ are of the form $V_{\OO_F} \left( \sum_i (w_i \cdot 0) \right)$ as we vary $\{ w_i \}_{i=1}^d$ in the Weyl group subject to the condition that $\sum_{i=1}^d l(w_i) = n$. In particular, $H^*(\mathfrak n, \Zp)$ is a finite, free $\Zp$-module.

Fix an unordered $d$-uple $\underline w = \{ w_i \}_{i=1}^d$ of Weyl group elements (possibly repeated) whose sum of lengths is $n$. The multiplicity of the module $V_{\OO_F} \left( \sum_i \left( w_i \cdot 0 \right) \right)$ in $H^n_{\Zp} \left( \mathfrak n, \Zp \right)$ is the number of $\Gal(F / \Qp)$-orbits on $C_{\underline w}$ where $C_{\underline w}$ is the set of all bijections $\Aut(F / \Qp) \lra \underline w$ and $\Gal(F/ \Qp)$ acts by precomposition.
\end{restatable}
We introduce $\Aut(F / \Qp)$ in the definition of the set $C_{\underline w}$ to emphasize that we are assigning a $w_i$ to each Galois element in $\Gal(F / \Qp)$, but we do not care about the group structure yet - until we consider the $\Gal(F / \Qp)$-orbits on $C_{\underline w}$.
\begin{rem} It may happen that different unordered $d$-uples $\{ w_i \}$ and $\{ w'_i \}$ give the same character $\sum_i (w_i \cdot 0) = \sum_i (w'_i \cdot 0)$ and therefore the two $\mathrm T$-modules are isomorphic. Understanding when that happens seems to be quite a difficult combinatorial problem which we have not attempted to solve.
\end{rem}

Next, we study the continuous group cohomology of $N = \mathrm U(\OO_F)$ by comparing it with the Lie algebra cohomology of $\mathfrak n$. More precisely, we follow Polo and Tilouine's blueprint in section 3 of \cite{PT} to obtain a spectral sequence relating Lie algebra cohomology and group cohomology. The notation and terminology related to filtered algebras are defined in section \ref{spectralsection}.
\begin{restatable}{thm}{equivariantss} \label{Lietogroupss} Consider the augmentation filtration on the completed group algebra $\Zp[[N]]$.
Let $V$ be a finitely generated $\Zp$-module with a continuous action of $\mathrm B(\OO_F)$. Suppose that as a filtered $T(\OO_F) \sharp \Zp[[N]]$-module, $V$ has a bounded and discrete filtration.

Then there exists a convergent spectral sequence of $\mathrm T(\OO_F)$-modules \[ E_1^{r,s} = H^{r+s}_{\gr \Zp[[N]]} \left( \gr \Zp, \gr V \right)_r  \Rightarrow E_{\infty}^{r,s} = H^{r+s}_{\Zp[[N]]} ( \Zp, V )_r = H^{r+s} (N, V)_r. \]
that converges to the graded $\mathrm T(\OO_F)$-module associated to a $\mathrm T(\OO_F)$-equivariant filtration of $H^*(N, V)$.
\end{restatable}
Finally, we show that for the trivial module $V = \Zp$, the spectral sequence converges at the first page, so that the previous results yield
\begin{restatable}{thm}{groupLie} \label{groupLiecomparison} There is a $\mathrm T(\OO_F)$-equivariant isomorphism \[ H^*_{\Zp} \left( \mathfrak n, \Zp \right) \cong \gr H^* \left( N, \Zp \right) \] between the Lie algebra cohomology of the $\Zp$-Lie algebra $\mathfrak n$ and the graded module associated to the $\mathrm T(\OO_F)$-equivariant filtration of $H^*(N, \Zp)$ induced by the augmentation ideal. 
\end{restatable}

The proofs of these theorems follow the outline of Polo and Tilouine's work (\cite{PT}, section 3): we filter the completed group algebra $\Zp[[N]]$ by powers of the augmentation ideal, show that the associated graded algebra is isomorphic to the enveloping algebra of $\mathfrak n$, and then prove that the filtered complex spectral sequences in theorem \ref{Lietogroupss} collapses on the first page, giving us the isomorphism of theorem \ref{groupLiecomparison}.
%

\subsection{Notation} \label{notationsection}
We collect in this section some notation and terminology used throughout the rest of the paper.

Let $\mathrm G$ be a connected, reductive, split $F$-group. We fix an $\OO_F$-model of $\mathrm G$ as in theorem 1.2 of \cite{Conrad2}, due to Demazure and Gabriel. This $\OO_F$-model (still often denoted $\mathrm G$) is a smooth, affine group scheme with connected, reductive fibers.

Since $\mathrm G$ is split (being a Chevalley group scheme of a split $F$-group), we can choose a maximal $\OO_F$-torus in it which is fiberwise split (see example 3.1 of \cite{Conrad2}), and we denote it by $\mathrm T$. Its generic fiber $\mathrm T_F$ is a maximal, split $F$-torus in $\mathrm G_F$ and all maximal, split $F$-tori of $\mathrm G_F$ are $F$-conjugate to it.

We consider next the roots $\Phi( \mathrm G, \mathrm T)$ of $\mathrm T$ on $\mathrm G$ (see sections 3-5 of \cite{Conrad} for the theory of root data over $\OO_F$-reductive groups).
Fix a cocharacter $\lambda \in \Hom_{\OO_F} \left( \mathbb G_m, \mathrm T \right)$ not annihilated by any root in $\Phi \left( \mathrm G, \mathrm T \right)$ - this defines a Borel subgroup $\mathrm B = \mathrm P_{\mathrm G}(\lambda)$ over $\OO_F$ (we use the dynamical description of parabolic subgroups as in \cite{Conrad}, see in particular sections 4.1 and 5.2), as in example 3.1 of \cite{Conrad2}.

As mentioned in the introduction, we will often abuse notation and suppress the fiber-subscript, so $\mathrm G$, $\mathrm B$ and $\mathrm T$ will be used to denote both the $\OO_F$-models and their generic fibers, the meaning being clear from context.

Our choice of the Borel subgroup $\mathrm B$ determines a basis of simple roots $\Delta = \{ \alpha_1, \ldots, \alpha_r \}$ as well as a choice of positive roots $\Phi^+ = \Phi^+(\mathrm G, \mathrm T)$ inside $\Phi(\mathrm G,\mathrm T)$.

For each positive root $\alpha \in \Phi^+(\mathrm G, \mathrm T)$, we define the height of $\alpha$ as \[ h(\alpha) = \sum_i n_i \textnormal{ where } \alpha = \sum_i n_i \alpha_i. \]

We denote by $\rho$ the half-sum of the positive roots $\Phi^+ $.
If $\mathrm G$ is simply connected then $\rho \in X^* \left( \mathrm T \right)$.
In this instance, the height $h(\alpha)$ is also equal to $\langle \alpha^{\vee}, \rho \rangle$ by proposition 29, section 1 of \cite{Bourbaki}, where ${\langle \cdot , \cdot \rangle: X_*(\mathrm T) \times X^* (\mathrm T) \lra \Z}$ is the usual perfect pairing between cocharacter and character lattice.

The Coxeter number of $\mathrm G$ is the maximum $h$ of $\langle \alpha^{\vee}, \rho \rangle +1$ as $\alpha$ varies among the positive roots.

We fix an ordering of the positive roots such that the height is non-decreasing, we have then $\Phi^+ = \{ \alpha_1, \ldots, \alpha_n\}$.
The Weyl group $W = N_{\mathrm G}(\mathrm T) / \mathrm T$ is generated by the simple reflections $\{ s_{\alpha} \}_{\alpha \in \Delta}$.
This gives the notion of length of a Weyl element $w \in W$: the smallest possible number of simple reflections needed to write $w$ as a word in the $s_{\alpha}$'s.

Recall the dot action of the Weyl group $W$ on the character lattice: \[ w \cdot \lambda = w(\lambda + \rho ) - \rho \qquad \forall w \in W, \lambda \in X^*(\mathrm T). \]

For each root $\alpha \in \Phi^+$ we have the associated root group $\mathrm U_{\alpha}$, an $\OO_F$-group scheme isomorphic to $\mathbb G_a$, and we fix isomorphisms $\theta_{\alpha}: \mathbb G_a \lra \mathrm U_{\alpha}$ giving rise to a Chevalley system, as explained for instance in proposition 6.3.4 and remark 6.3.5 of \cite{Conrad}.

As theorem 5.1.13 in \cite{Conrad} explains, the multiplication map is then an isomorphism of $\OO_F$-schemes: \begin{equation} \label{multiplicationisom} \prod_{\alpha \in \Phi^+} \mathrm U_{\alpha} \lra \mathrm U \end{equation} where the product of the root groups is ordered accordingly to the ordering of $\Phi^+$ we have fixed.

Denote $\mathfrak u_{\alpha} = \Lie \mathrm U_{\alpha}$ the $\OO_F$-Lie algebra of the root group $\mathrm U_{\alpha}$.

For technical purposes, it is convenient to bring down our setup to $\Zp$.
It is known (see for instance \cite{oesterle}, proposition A.3.7) that Weil restriction along the finite, free map $\Zp \hookrightarrow \OO_F$ preserves split unipotent group schemes.
Denote then $\mathrm U' = \Res{\OO_F}{\Zp}{\mathrm U}$ the Weil restriction of $\mathrm U$ to $\Zp$ (a split, unipotent group scheme over $\Zp$) and $\mathfrak n = \Res{\OO_F}{\Zp}{\mathfrak u}$ its $\Zp$-Lie algebra.\footnote{Our notion for Weil restriction of Lie algebras follows that of Oesterle in \cite{oesterle}, proposition A.3.3: of the $\OO_F$-module $\mathfrak g$ we only remember its structure as a $\Zp$-module, and the bracket operation is also seen as a $\Zp$-bilinear map.}
Similarly, for each root $\alpha$ we let $\mathrm U'_{\alpha} = \Res{\OO_F}{\Zp}{\mathrm U_{\alpha}}$ and $\mathfrak n_{\alpha} = \Res{\OO_F}{\Zp}{\mathfrak u_{\alpha}}$ be the Weil restrictions of the root group corresponding to $\alpha$ and its Lie algebra.
The conjugation action of $\mathrm T$ on $\mathrm U$ and the induced action on $\mathfrak u$ are algebraic, and thus can be brought down to $\Zp$ to get algebraic actions of the (non-split) torus $\Res{\OO_F}{\Zp} \mathrm T$ on $\mathrm U'$ and $\mathfrak n$ defined over $\Zp$.

Finally, denote by $N = \mathrm U'(\Zp) \cong \mathrm U(\OO_F)$ the integral points of $\mathrm U$, and similarly $N_{\alpha} = \mathrm U'_{\alpha}(\Zp) \cong \mathrm U_{\alpha}(\OO_F)$.

When computing Lie algebra cohomology, we will use a subscript to remind the reader what linear structure we are using: for instance $H^*_{\Zp} \left( \mathfrak n, - \right)$ is the cohomology of the $\Zp$-Lie algebra $\mathfrak n$ while $H^*_{\OO_F} \left( \mathfrak u, - \right)$ is the cohomology of the $\OO_F$-Lie algebra $\mathfrak u$.

\subsection{Further directions and applications}
This paper arose from the goal of describing $H^* \left( \mathrm U (\OO_F) , \Z / p^n \right)$ as a $\mathrm T(\OO_F)$-module, and in particular when the cohomology groups $H^a \left( \mathrm T(\OO_F), H^b \left( \mathrm U (\OO_F) , \Z / p^n \right) \right)$ vanish.
This understanding was crucial to extend the results of the author's PhD thesis to a larger generality. On the other hand, the results of this paper are independent of the rest of the author's thesis and are thus written separately.

It is clear that the results and the proofs in this paper go through in the more general setup when $\mathrm B$ is replaced by a parabolic subgroup $\mathrm P$, $\mathrm U$ is replaced by the unipotent radical of the parabolic $\mathcal R_u ( \mathrm P)$ and the torus $\mathrm T$ is replaced by a Levi subgroup $\mathrm L \cong \mathrm P / \mathcal R_u(\mathrm P)$.
Indeed, the sources \cite{PT,vigre,FP} work in this greater generality of a parabolic subgroup $\mathrm P$, and the work of Hartley \cite{hartley1,hartley2} can immediately be adapted to a general unipotent radical $\mathcal R_u(\mathrm P)$.
One will then get a description of $H^* \left( \mathcal R_u(\mathrm P) (\OO_F), \Zp \right)$ as a $\mathrm L(\OO_F)$-module. We have not made this result explicit simply because understanding $H^a \left( \mathrm L(\OO_F), H^b \left( \mathcal R_u (\mathrm P) (\OO_F) , \Z / p^n \right) \right)$ is much harder whenever $\mathrm L$ is not a torus, and thus the intended original application of this paper's results to the author's thesis does not follow through.

The sources \cite{PT,vigre,FP,GK} deal with the cohomology of $\mathrm U$ with coefficients in a particular class of highest weight modules (and not just trivial coefficients).
It would be interesting to pursue this line of investigation, and understand how much of the present paper can be carried through to this more general setup of coefficients in a highest weight module with nontrivial $\mathrm U$-action.
While Kostant's theorem \ref{kostantthm} and the results of section \ref{finalsection} (in particular the existence of the equivariant spectral sequence of theorem \ref{Lietogroupss}) should be amenable to this general setup, it is not immediately clear how to replace the step in the proof of corollary \ref{kostantcor} that uses Kunneth's theorem. \\

The interest in the cohomology of analytic $p$-adic groups, for example the ones considered in this paper, has spiked in recent years due to their connection to various area of representation theory and number theory, such as $\mod p$ and $p$-adic representation theory of reductive $p$-adic groups.

One potential application of the computations in this paper is the following: in \cite{GK}, Grosse-Klonne studies universal modules $M_{\chi} (V)$ for $\bmod p$ spherical Hecke algebras $\mathcal H(G,K,V)$ of $p$-adic groups.
Among other results, he describes some sufficient conditions for the freeness of the universal modules: these conditions are precisely expressed in terms of the cohomology $H^* \left( N, V \otimes_{\Zp} \FFp \right)$ (see proposition 6.4 and theorem 8.2 in \cite{GK}, which consider the special case $F = \Qp$).
It is conceivable that by combining corollary \ref{kostantcor} and theorem \ref{groupLiecomparison} one gets a description of $H^*(N, \Zp)$ (and more generally of $H^*(N,V)$ for a special class of highest weight modules $V$, if the technical obstacle explained above can be removed) and hence of $H^*(N, \FFp)$ (resp. $H^*(N, V \otimes_{\Zp} \FFp )$), explicit enough to extend the aforementioned results in \cite{GK} to the case of a general $p$-adic field $F$.

The freeness of the universal module over its Hecke algebra would have many important consequences, for instance the existence of supersingular representations for $\mathrm G(F)$ over $\FFp$ (see remark 5b in the introduction of \cite{GK}). This was an open problem for a general $p$-adic reductive group in characteristic $p$ and has only very recently been settled by Vigneras in \cite{vigneras} besides a few special cases. Vigneras' approach is through a detailed and careful study of Hecke algebras via their presentations, and it would be interesting to have a different proof of existence.

\subsection{Acknowledgments}
This paper owes a debt of gratitude to my advisor Akshay Venkatesh, who encouraged me to pursue this project to extend my thesis's results to a larger generality and suggested the correct strategy to approach the problem. I also want to thank Nivedita Bhaskhar, Rita Fioresi and Mihalis Savvas for helpful conversations, and an anonymous referee for suggesting some improvements.

\section{Lie algebra cohomology}
In this section we compute the Lie algebra cohomology of $\mathfrak u$ and of $\mathfrak n$ with coefficients in $\OO_F$ and $\Zp$ respectively. This is based on a version of Kostant's theorem for unipotent $\OO_F$-Lie algebras very similar to those proved by Polo and Tilouine in \cite{PT} and by the authors of \cite{vigre}.

\subsection{Reduction steps} \label{reductionsteps}
We start with some reductions: we are interested in the $\Res{\OO_F}{\Zp} \mathrm T(\Zp) = \mathrm T(\OO_F)$-action on $H^*_{\Zp} (\mathfrak n, \Zp)$.
Since the action is via conjugation, it certainly factors through the center $\mathrm Z_{\mathrm G}$.
Moreover, when we pass from $\mathrm G$ to its adjoint quotient $\mathrm G / \mathrm Z_{\mathrm G}$, the unipotent radical maps isomorphically onto its image, so the same is true for its Lie algebra.
We can thus assume that $\mathrm G$ is semisimple and adjoint.
In fact, the same reasoning allows us to consider any element of the central isogeny class of $\mathrm G$, since maximal tori also correspond to one another under the central isogeny $\mathrm G^{\mathrm {sc}} \lra \mathrm G^{\ad}$. For example, we could equivalently assume that $\mathrm G$ is semisimple and simply connected.

Recall that a connected, reductive $F$-group is said to be \emph{simple} if it does not have any nontrivial smooth proper connected normal subgroup.
As explained in corollary 10.1.3 of \cite{CF}, the multiplication map \[ \prod_i \mathrm G_i \lra \mathrm G \] from the product of the simple factors of $\mathrm G$ to $\mathrm G$ itself is a central isogeny, and it is an isomorphism if $\mathrm G$ is simply connected or adjoint (which we are free to assume).
Moreover, if $\mathrm G$ is simply connected (resp. adjoint), then each of the $\mathrm G_i$ is also simply connected (resp. adjoint), and maximal split tori, Borel subgroups and their unipotent radicals correspond under the product map.

In particular, $\prod \mathrm T_i \cong \mathrm T$ where each $\mathrm T_i$ is a maximal split torus of $\mathrm G_i$ and $\prod \mathrm U_i \cong \mathrm U$ where $\mathrm U_i$ is the unipotent radical of the Borel subgroup $\mathrm B_i = \mathrm G_i \cap \mathrm B$ of $\mathrm G_i$.

For the Lie algebras, we obtain that $\prod_i \mathfrak u_i \cong \mathfrak u$ where the isomorphism is equivariant for the action of $\prod \mathrm T_i(\OO_F) \cong \mathrm T(\OO_F)$ - the same clearly holds for the Weil restrictions $\prod_i \mathfrak n_i \cong \mathfrak n$ with equivariant $\prod \Res{\OO_F}{\Zp} \mathrm T_i (\Zp) \cong \Res{\OO_F}{\Zp} \mathrm T (\Zp)$-actions.
The Kunneth formula gives then \[ H^n_{\Zp} \left( \mathfrak n, \Zp \right) \cong \bigoplus_{\sum k_i = n} \bigotimes_i H^{k_i}_{\Zp} \left( \mathfrak n_i, \Zp \right), \] where obviously the actions of $\Res{\OO_F}{\Zp} \mathrm T(\Zp)$ on the left and of $\prod \Res{\OO_F}{\Zp} \mathrm T_i(\Zp)$ on the right correspond.
This shows that for the purpose of understanding the $\Res{\OO_F}{\Zp} \mathrm T(\Zp)$-action on $H^*_{\Zp} \left( \mathfrak n , \Zp \right)$ we can assume that $\mathrm G$ is simple.

Moreover, the Kunneth formula also gives \[ H^n \left( N, \Zp \right) \cong \bigoplus_{\sum k_i = n} \bigotimes_i H^{k_i} \left( N_i, \Zp \right) \] where we denote $N_i = \Res{\OO_F}{\Zp}{ \mathrm U_i} (\Zp) = \mathrm U_i(\OO_F)$, and again the isomorphism is equivariant for the action of $\Res{\OO_F}{\Zp} \mathrm T(\Zp) \cong \prod \Res{\OO_F}{\Zp} \mathrm T_i(\Zp)$.

Therefore, it suffices to prove theorem \ref{groupLiecomparison} in case of $\mathrm G$ simple and simply connected, and the case of general connected reductive split $\mathrm G$ follows from the above reasoning.

\begin{lem} \label{Liecohomologybasechange} Let $\mathfrak g$ be a $\Zp$-Lie algebra. Let $F / \Qp$ be a finite extension. Denote $\mathfrak g_{\OO_F} = \mathfrak g \otimes_{\Zp} \OO_F$ the `base change' of $\mathfrak g$ to $\OO_F$.
Let $V$ be a $\mathfrak g$-module which is a finite, free $\Zp$-module.
Then \[ H^*_{\OO_F} \left( \mathfrak g_{\OO_F}, V \otimes_{\Zp} \OO_F \right) \cong H^*_{\Zp} \left( \mathfrak g, V \right) \otimes_{\Zp} \OO_F. \]
In particular, if $F / \Qp$ is Galois, then \[ H^*_{\OO_F} \left( \mathfrak g_{\OO_F}, V \otimes_{\Zp} \OO_F \right)^{\Gal(F / \Qp)} \cong H^*_{\Zp} \left( \mathfrak g, V \right). \]
\end{lem}
\begin{proof} Recall that for any free $\OO_F$-module $W$, the cohomology of $\mathfrak g_{\OO_F}$ with coefficients in $W$ is defined as the homology of the complex \[ C^{\bullet} (\mathfrak g_{\OO_F}, W) : = \Hom_{\OO_F} \left( \bigwedge^{\bullet} \mathfrak g_{\OO_F}, W \right). \]
The differential is defined as \[ df (x_0, \ldots, x_q) = \sum_{0 \le i < j \le q} (-1)^{i+j} f \left( [x_i, x_j], x_0, \ldots, \hat {x_i}, \ldots, \hat {x_j}, \ldots, x_q \right) + \] \[ + \sum_{i=0}^q (-1)^i x_i. f \left( x_0, \ldots, \hat{x_i}, \ldots, x_q \right). \]
Since every $\bigwedge^n \mathfrak g$ is a free $\Zp$-module, we have natural isomorphisms \[ \Hom_{\Zp} \left( \bigwedge^n \mathfrak g, V \right) \otimes_{\Zp} \OO_F \cong \Hom_{\OO_F} \left( \left( \bigwedge^n \mathfrak g \right) \otimes_{\Zp} \OO_F, V \otimes_{\Zp} \OO_F \right), \] where we also notice that since $\OO_F$ is a free $\Zp$-module, we have $\left( \bigwedge^n \mathfrak g \right) \otimes_{\Zp} \OO_F \cong \bigwedge^n \left( \mathfrak g_{\OO_F} \right)$ in a natural way.

In fact, these isomorphisms are compatible with the differentials, where on the left hand side we have $d \otimes \id_{\OO_F}$ and on the right hand side $f \otimes x \mapsto df \otimes x$, so that we obtain an isomorphisms of complexes \begin{equation} \label{homtensor} C^{\bullet} ( \mathfrak g, V) \otimes_{\Zp} \OO_F \cong C^{\bullet} \left( \mathfrak g_{\OO_F}, V \otimes_{\Zp} \OO_F \right) \end{equation}

Using that $\OO_F$ is finite, free as a $\Zp$-module, and hence flat, the universal coefficient theorem as in \cite{rotman} (corollary 7.56) guarantees that we obtain isomorphisms \begin{equation} \label{basechangeLiecohomology} H^n_{\Zp} \left( \mathfrak g, V \right) \otimes_{\Zp} \OO_F \cong H^n_{\OO_F} \left( \mathfrak g_{\OO_F}, V \otimes_{\Zp} \OO_F \right) \end{equation} as the relevant $\Tor_1^{\Zp} \left( H^{n-1}(\mathfrak g, V), \OO_F \right)$ group is always zero.
This proves the first part of the claim.

Suppose now that $F / \Qp$ is Galois. $\Gal (F / \Qp)$ acts on the complex \[ \Hom_{\Zp} \left( \bigwedge^{\bullet} \mathfrak g, V \right) \otimes_{\Zp} \OO_F \] on the second factor, and through the isomorphism of complexes in formula \ref{homtensor} we obtain a Galois action on $\Hom_{\OO_F} \left( \bigwedge^{\bullet} \mathfrak g_{\OO_F} , V \otimes_{\Zp} \OO_F \right)$ defined as \[ (\gamma.f)(x_1, \ldots, x_q) : = \gamma \left( f \left( \gamma^{-1}.x_1, \ldots, \gamma^{-1}.x_q \right) \right) \qquad \forall \gamma \in \Gal(F/ \Qp). \]
Since the Galois action on $\Hom_{\Zp} \left( \bigwedge^{\bullet} \mathfrak g, V \right) \otimes_{\Zp} \OO_F$ obviously commutes with the differential, the same holds for this Galois action on $\Hom_{\OO_F} \left( \bigwedge^{\bullet} \mathfrak g_{\OO_F} , V \otimes_{\Zp} \OO_F \right)$ and thus it descends to an action on the cohomology $H^*_{\OO_F} \left( \mathfrak g_{\OO_F}, V \otimes_{\Zp} \OO_F \right)$.

In particular, the isomorphism in formula \ref{basechangeLiecohomology} respects the Galois action, where on the left hand side this Galois action is only on the second tensor factor. The second claim of the lemma follows immediately.
\end{proof}

\subsection{Kostant's theorem} \label{kostantsubsection}
Our goal in this section is to prove theorem \ref{kostantthm} which we recall for the reader's convenience.
\kostant*

We start building towards the proof of this theorem by doing some reductions as well as introducing some technical tools. We follow the ideas in \cite{vigre}, with considerable simplifications since we are only considering the case of the unipotent radical of a Borel subgroup, and not of a general parabolic subgroup.

By the classification of simple algebraic groups and work of Chevalley, we know that $\mathrm G$, as well as $\mathfrak g$ and $\mathfrak u$, have $\Zp$-models since their structure constants can be defined over $\Z$.
For all highest weight $\mathrm T$-modules we have \[ V_{\OO_F} (w \cdot 0) \cong V_{\Zp} (w \cdot 0) \otimes_{\Zp} \OO_F, \] hence it is clear thanks to lemma \ref{Liecohomologybasechange} that it suffices to prove the theorem for $F = \Qp$.

Fix a $\Zp$-basis of weight vectors $\{ x_{\alpha} \}_{\alpha \in \Phi^+}$ for $\mathfrak u$, and let $\{ f_{\alpha} \}_{\alpha \in \Phi^+}$ be the dual $\Zp$-basis of $\mathfrak u^*$.
Notice that the cohomology of $\mathfrak u$ can equivalently be computed by the complex $\bigwedge^{\bullet} (\mathfrak u^*)$, and that a basis of $\bigwedge^k (\mathfrak u^*)$ is given by \[ f_{\underline{\alpha}} = f_{\alpha_1} \wedge \ldots \wedge f_{\alpha_k} \] as $\alpha_1, \ldots, \alpha_k$ vary in $\Phi^+$.

Given a subset $\Psi \subset \Phi^+$, we define \[ \langle \Psi \rangle := \sum_{\beta \in \Psi} \beta \] and for $w \in W$ we define \[ \Phi(w) = w \Phi^- \cap \Phi^+. \]
\begin{lem} \label{rootlemma} Let $w \in W$. \begin{enumerate}
    \item The size of $\Phi(w)$ is the length $l(w)$.
    \item $w \cdot 0 = - \langle \Phi(w) \rangle$.
    \item If $w \cdot 0 = - \langle \Psi \rangle$ for some $\Psi \subset \Phi^+$, then $\Psi = \Phi(w)$.
\end{enumerate}
\end{lem}
\begin{proof} Part 1 is corollary 1.7 in \cite{humphreys}, part 2 is proposition 3.19 in \cite{knapp} and part 3 is lemma 3.1.2 in \cite{vigre}.
\end{proof}

We make explicit the following result, explained in section 3.2 of \cite{vigre}. Let $w \in W$, so that by lemma \ref{rootlemma} $w \cdot 0 = - \langle \Phi(w) \rangle$ and $\Phi(w)$ is the unique subset of positive roots satisfying that equality. Enumerate $\Phi(w) = \{ \beta_1, \ldots, \beta_n \}$.
\begin{lem} \label{highestweightinkostant} The vector \[ f_{\Phi(w)} : = f_{\beta_1} \wedge \ldots \wedge f_{\beta_n} \] has weight $w \cdot 0$ in $\bigwedge^n \left( \mathfrak u^* \right)$, and in fact spans (over $\Zp$) the relevant eigenspace. Moreover, it descends to a nonzero element of $H^n \left( \mathfrak u, \Zp \right)$ and $n$ is the only degree where the weight $w \cdot 0$ appears.
\end{lem}
\begin{proof}
It is clear that $f_{\Phi(w)}$ has weight exactly $ - \langle \Phi(w) \rangle = w \cdot 0$, where the minus sign appears because $\mathrm T$ is acting on the dual Lie algebra.
Any weight in $\bigwedge^{\bullet} \left( \mathfrak u^* \right)$ has to be a sum of negative roots, so lemma \ref{rootlemma} guarantees that if $w \cdot 0$ appears as a weight, such a sum should be $- \langle \Phi(w) \rangle$, which on the other hand can only appear in degree $| \Phi(w)| = l(w)$ (by lemma \ref{rootlemma}). Finally, it is clear thanks to the alternating property that any other `arrangement' of the roots $\{ \beta_1, \ldots, \beta_n \} = \Phi(w)$ to create a vector in $\bigwedge^n \left( \mathfrak u^* \right)$ has to be a multiple of $f_{\Phi(w)}$.

Now, since the differentials in the complex $\bigwedge^{\bullet} \left( \mathfrak u^* \right)$ are $\mathrm T$-equivariant, they preserve the weights and hence $f_{\Phi(w)}$ must descend to a cohomology class in $H^n \left( \mathfrak u, \Zp \right)$ - indeed the weight $w \cdot 0$ does not appear in $\bigwedge^{n+1} \left( \mathfrak u^* \right)$ and hence $f_{\Phi(w)}$ is a cocycle, but the weight $w \cdot 0$ does not appear in $\bigwedge^{n-1} \left( \mathfrak u^* \right)$ and hence $f_{\Phi(w)}$ is not a coboundary.

The fact that the weight $w \cdot 0$ only appears in degree $n$ in $H^*_{\Zp} \left( \mathfrak u, \Zp \right)$ is a consequence of the same fact for the complex $\bigwedge^{\bullet} \left( \mathfrak u^* \right)$, since the Lie algebra cohomology $H^*_{\Zp}(\mathfrak u, \Zp)$ is a $\mathrm T$-equivariant subquotient of it.
\end{proof}

\begin{proof}[Proof of theorem \ref{kostantthm}]
Lemma \ref{highestweightinkostant} shows that we have an injection \[ i: \bigoplus_{w \in W, l(w) = n} V_{\Zp}(w \cdot 0) \hookrightarrow H^n_{\Zp}( \mathfrak u, \Zp) \] of a finite, free $\Zp$-module into a finitely generated $\Zp$-module.

To show that this injection is an equality, it suffices to prove that $i \otimes_{\Zp} \overline{\Qp}$ and $i \otimes_{\Zp} \FFp$ are isomorphisms.

Now, theorem 4.1.1 in \cite{vigre} constructs the same injection for $\C$ and $\FFp$-coefficients, and proves that the injection is indeed an isomorphism.
The proof with $\C$-coefficients works verbatim with $\overline {\Qp}$-coefficients, since it simply uses that $\C$ is an algebraically closed field of characteristic zero.
The universal coefficient theorem (applied in the same fashion as in formula 2.1.2 in \cite{vigre}) yields that \[ H^n \left( \mathfrak u_{\overline {\Qp}}, \overline{\Qp} \right) \cong H^n \left( \mathfrak u, \Zp \right) \otimes_{\Zp} \overline{ \Qp} \] since $\overline \Qp$ is a divisible $\Zp$-module.
%

As for the Weyl modules $V_{\Zp} (\lambda)$, it is immediate that $V_{\Zp} (\lambda) \otimes_{\Zp} \overline {\Qp}$ is the highest weight $\overline {\Qp}$-module of highest weight $\lambda$. This fact and theorem 4.1.1 in loc. cit. show that $i \otimes_{\Zp} \overline \Qp$ is indeed an isomorphism.

It remains to check that $i \otimes_{\Zp} \FFp$ is an isomorphism. Since theorem 4.1.1 is proven with $\FFp$-coefficients, we only need to show that $i \otimes_{\Zp} \FFp$ coincides with the map of theorem 4.1.1 of loc. cit.

For the Weyl modules, it is again obvious that $V_{\Zp}(w \cdot 0) \otimes_{\Zp} \FFp \cong V_{\FFp}(w \cdot 0)$, as both sides are $1$-dimensional $\FFp$-modules on which $\mathrm T(\Zp)$ acts as $w \cdot 0$. 

To show that $H^n \left( \mathfrak u , \Zp \right) \otimes_{\Zp} \FFp \cong H^n \left( \mathfrak u_{\FFp}, \FFp \right)$ it suffices to show that $H^{n-1} \left( \mathfrak u, \Zp \right)$ is a free $\Zp$-module, and then the universal coefficient theorem (\cite{rotman}, corollary 7.56) will yield the required isomorphism.
This is clear by induction on $n$: the base case $n=0$ has simply $H^0_{\Zp} \left( \mathfrak u, \Zp \right) \cong \Zp$, while the inductive step uses that we have proven the theorem for $n-1$, so that $H^{n-1} (\mathfrak u, \Zp) \cong \bigoplus_{w \in W, l(w) =n-1} V_{\Zp} (w \cdot 0)$ is a finite, free $\Zp$-module.
\end{proof}

We can now take Kostant's theorem \ref{kostantthm} as a starting point, and compute the cohomology of $\mathfrak n$, a $\Zp$-Lie algebra, with coefficients in the trivial module $\Zp$. Notice that we cannot immediately apply Kostant's theorem to $\mathfrak n = \Res{\OO_F}{\Zp}{\mathfrak u}$, since this is not (in general) the Lie algebra of the unipotent radical of a Borel subgroup of a split group.

Suppose that $F / \Qp$ is Galois, then lemma \ref{Liecohomologybasechange} gives that \[ H^*_{\Zp} \left( \mathfrak n, \Zp \right) \cong H^*_{\OO_F} \left( \mathfrak n_{\OO_F}, \OO_F \right)^{\Gal(F/ \Qp)}, \] so we aim to compute the right hand side. Because of the technical nature of the following argument, for the sake of clarity we drop the shortcut notation $\mathfrak u$, $\mathfrak n$ and use instead $\Lie \mathrm U$, $\Lie \left( \Res{\OO_F}{\Zp} \mathrm U \right)$ and so on.

As explained in \cite{CGP}, appendix A.7 we have an isomorphism \[ \Lie (\mathrm U) \cong \Lie \left( \Res{\OO_F}{\Zp} \mathrm U \right) \] of $\Zp$-Lie algebras, and we also have (by the discussion after formula 2.1.3 in \cite{CGP}) \[ \Lie \left( \Res{\OO_F}{\Zp} \mathrm U \right) \otimes_{\Zp} \OO_F \cong \Lie \left( \left( \Res{\OO_F}{\Zp} \mathrm U \right) \times_{\Zp} \OO_F \right) \] and since \[ \left( \Res{\OO_F}{\Zp} \mathrm U \right) \times_{\Zp} \OO_F \cong \prod_{\sigma \in \Gal(F/ \Qp)} \mathrm U \times_{\OO_F, \sigma} \OO_F \] we conclude that \[ \Lie \left( \Res{\OO_F}{\Zp} \mathrm U \right) \otimes_{\Zp} \OO_F \cong \prod_{\sigma \in \Gal(F/ \Qp)} \Lie \left( \mathrm U \times_{\OO_F, \sigma} \OO_F \right). \]
These isomorphisms are compatible for the action of \[ \left( \Res{\OO_F}{\Zp} \mathrm T \right) \times_{\Zp} \OO_F \cong \prod_{\sigma \in \Gal(F/ \Qp)} \mathrm T \times_{\OO_F, \sigma} \OO_F, \] in the sense that the left side acts naturally on $\Lie \left( \Res{\OO_F}{\Zp} \mathrm U \times_{\Zp} \OO_F \right)$ and the right side acts componentwise on $\prod_{\sigma \in \Gal(F/ \Qp)} \Lie \left( \mathrm U \times_{\OO_F, \sigma} \OO_F \right)$.

Therefore when taking cohomology, we obtain the isomorphism \[ H^*_{\OO_F} \left( \Lie \left( \Res{\OO_F}{\Zp} \mathrm U \right) \otimes_{\Zp} \OO_F, \OO_F \right) \cong H^*_{\OO_F} \left( \prod_{\sigma \in \Gal(F/ \Qp)} \Lie \left( \mathrm U \times_{\OO_F, \sigma} \OO_F \right), \OO_F \right) \] and by Kunneth's theorem the latter module is isomorphic to \[ \bigotimes_{\sigma \in \Gal(F/ \Qp)} H^*_{\OO_F} \left( \Lie \left( \mathrm U \times_{\OO_F, \sigma} \OO_F \right), \OO_F \right) \] where the action of $\prod_{\sigma \in \Gal(F/ \Qp)} \mathrm T \times_{\OO_F, \sigma} \OO_F$ is again component-wise.

The Galois group $\Gal(F/\Qp)$ acts on both $\Res{\OO_F}{\Zp} \mathrm T \times_{\Zp} \OO_F$ and $\Res{\OO_F}{\Zp} \mathrm U \times_{\Zp} \OO_F$ on the $\OO_F$-factor, compatibly with the conjugation action of the first group on the second (because the conjugation action is algebraic). In particular, this implies that the torus action on the cohomology of $\Lie \left( \Res{\OO_F}{\Zp} \mathrm U \times_{\Zp} \OO_F \right)$ is Galois-semilinear - see also formula \ref{galoistorusequivariance} afterwards.

Therefore, the Galois-fixed submodule of $H^*_{\OO_F} \left( \Lie \left( \Res{\OO_F}{\Zp} \mathrm U \right) \otimes_{\Zp} \OO_F, \OO_F \right)$ - which coincide with $H^*_{\Zp} \left( \Lie \left( \Res{\OO_F}{\Zp} \mathrm U \right), \Zp \right)$ by lemma \ref{Liecohomologybasechange} and the isomorphisms described above - is acted upon by the Galois-fixed subtorus of $\Res{\OO_F}{\Zp} \mathrm T \times_{\Zp} \OO_F$, which is just $\Res{\OO_F}{\Zp} \mathrm T$. In fact, taking Galois-invariants (on both the torus and the cohomology module) recovers the natural action of $\Res{\OO_F}{\Zp} \mathrm T$ on $H^*_{\Zp} \left( \Lie \left( \Res{\OO_F}{\Zp} \mathrm U \right), \Zp \right)$ which is exactly what we aim to compute.

We start by noticing that we can apply Kostant's theorem \ref{kostantthm} to each $\sigma$-twist, and obtain that the $\mathrm T \times_{\OO_F, \sigma} \OO_F$-module $H^{n_{\sigma}}_{\OO_F} \left( \Lie \left( \mathrm U \times_{\OO_F, \sigma} \OO_F \right), \OO_F \right)$ is a direct sum of characters indexed by Weyl-group elements of length $n_{\sigma}$: \[ H^n_{\OO_F} \left( \Lie \left( \mathrm U \times_{\OO_F, \sigma} \OO_F \right), \OO_F \right) \cong \bigoplus_{l(w_{\sigma})=n_{\sigma}} V_{\OO_F} \left( w_{\sigma} \cdot 0 \right) \] where the $\sigma$-subscript simply reminds us which Galois-twist we are looking at. This and the discussion above show, in particular, that $H^*_{\OO_F} \left( \Lie \left( \Res{\OO_F}{\Zp} \mathrm U \right) \otimes_{\Zp} \OO_F, \OO_F \right)$ is a finite, free $\OO_F$-module.

Assume first that $| \Gal(F / \Qp)| = d$ is coprime to $p$. Then the averaging operator \begin{equation} \label{projectionidempotent} e:= \frac{1}{d} \sum_{\gamma \in \Gal(F / \Qp)} \gamma : H^*_{\OO_F} \left( \Lie \left( \Res{\OO_F}{\Zp} \mathrm U \right) \otimes_{\Zp} \OO_F, \OO_F \right) \lra H^*_{\OO_F} \left( \Lie \left( \Res{\OO_F}{\Zp} \mathrm U \right) \otimes_{\Zp} \OO_F, \OO_F \right) \end{equation} is an idempotent projection on the Galois-fixed submodule of $H^*_{\OO_F} \left( \Lie \left( \Res{\OO_F}{\Zp} \mathrm U \right) \otimes_{\Zp} \OO_F, \OO_F \right)$.

The semilinearity of the Galois action means that for each $f \in H^*_{\OO_F} \left( \Lie \left( \mathrm U \times_{\OO_F, \sigma} \OO_F \right), \OO_F \right)$, $t \in \mathrm T \times_{\OO_F, \sigma} \OO_F$ and $\gamma \in \Gal(F/\Qp)$ we have \begin{equation} \label{galoistorusequivariance} \gamma. \left( t.f \right) = \gamma(t). \gamma(f) \end{equation} where $\gamma(t) \in \mathrm T \times_{\OO_F, \gamma \sigma} \OO_F$ and $\gamma(f) \in H^*_{\OO_F} \left( \Lie \left( \mathrm U \times_{\OO_F, \gamma \sigma} \OO_F \right), \OO_F \right)$.

In particular, if $\lambda \in X^*(\mathrm T)$ and $f$ is in the $\lambda$-eigenspace for $\mathrm T \times_{\OO_F, \sigma} \OO_F$ we have that \[ \gamma(t). \gamma(f) = \gamma. \left( t.f \right) = \gamma. \left( \lambda(t) \cdot f \right) = \gamma \left( \lambda(t) \right) \cdot \gamma(f) = \lambda \left( \gamma(t) \right) \gamma(f) \] where the last equality, where we swap $\gamma$ and $\lambda$, holds since $\lambda$ is an algebraic character.
This shows that $\gamma(f)$ is in the $\lambda$-eigenspace for $\mathrm T \times_{\OO_F, \gamma \sigma} \OO_F$.

We compute the Galois-fixed submodule of $H^*_{\OO_F} \left( \Lie \left( \Res{\OO_F}{\Zp} \mathrm U \right) \otimes_{\Zp} \OO_F, \OO_F \right)$ by computing the image of a basis of $H^*_{\OO_F} \left( \Lie \left( \Res{\OO_F}{\Zp} \mathrm U \right) \otimes_{\Zp} \OO_F, \OO_F \right)$ under the projection operator $e$.
Fix then a cohomological degree $n$, as well as \[ f_{\underline w} = \otimes f_{w_{\sigma}} \in \bigotimes_{\sigma \in \Gal(F / \Qp)} \OO_F \left( w_{\sigma} \cdot 0 \right) \subset \bigotimes_{\sigma \in \Gal(F / \Qp)} H^{l(w_{\sigma})}_{\OO_F} \left( \Lie \left( \mathrm U \times_{\OO_F, \sigma} \OO_F \right), \OO_F \right) \] where $\underline w = \left( w_{\sigma} \right)_{\sigma \in \Gal(F/ \Qp)}$ remembers the eigenspace to which $f_{\underline w}$ belongs under the torus action of $\prod_{\sigma \in \Gal(F/ \Qp)} \mathrm T \times_{\OO_F, \sigma} \OO_F$, and $\sum_{\sigma} l(w_{\sigma}) = n$.

Then equation \ref{galoistorusequivariance} and the discussion following it shows that for each Galois element $\gamma$, the cohomology class $\gamma.f$ belongs to the $\prod_{\sigma} \mathrm T \times_{\OO_F, \sigma} \OO_F$-eigenspace where each $\mathrm T \times_{\OO_F, \gamma \sigma} \OO_F$ acts via $w_{\sigma} \cdot 0$ -  that is, letting $t = \left( t_{\sigma} \right)_{\sigma} \in \prod_{\sigma} \mathrm T \times_{\OO_F, \sigma} \OO_F$, we obtain then \begin{equation} \label{extendedtorusaction} t. \left( \gamma.f_{\underline w} \right) = \prod_{\sigma} (w_{\sigma} \cdot 0) \left( t_{\gamma \sigma} \right) ( \gamma.f_{\underline w} ) \end{equation}

Therefore, if we apply the idempotent $e$ we obtain that \[  t. \left( e.f_{\underline w} \right) = t. \left( \frac{1}{d} \sum_{\gamma \in \Gal (F/\Qp)} \gamma.f_{\underline w} \right) = \frac{1}{d} \sum_{\gamma \in \Gal (F / \Qp)}  t. (\gamma.f_{\underline w}) = \frac{1}{d} \sum_{\gamma \in \Gal (F / \Qp)} \left( \prod_{\sigma} (w_{\sigma} \cdot 0) \left( t_{\gamma \sigma} \right) ( \gamma.f ) \right). \] 

We now describe how the Galois-invariant subtorus of $\left( \Res{\OO_F}{\Zp} \mathrm T \right) \times_{\Zp} \OO_F$ embeds into the isomorphic $\prod_{\sigma} \mathrm T \times_{\OO_F, \sigma} \OO_F$. Let $A$ be the $\Zp$-algebra of the affine group $\Res{\OO_F}{\Zp} \mathrm T$, then the Galois action on $\Res{\OO_F}{\Zp} \mathrm T \times_{\Zp} \OO_F$ induces, at the level of $\OO_F$-points, an action on $\left( \Res{\OO_F}{\Zp} \mathrm T \times_{\Zp} \OO_F \right) (\OO_F) = \Hom_{\OO_F \textnormal{-alg}} \left( A \otimes_{\Zp} \OO_F , \OO_F \right)$ given by \[ f \mapsto \left( a \otimes x \stackrel{\gamma}{\mapsto} \gamma \left( f \left( a \otimes \gamma^{-1}(x) \right) \right) \right). \]
For each $\OO_F$-algebra $B$ we have a natural adjunction isomorphism $\Hom_{\OO_F \textnormal{-alg}} \left( A \otimes_{\Zp} \OO_F , B \right) \cong \Hom_{\Zp \textnormal{-alg}} \left( A , B \right)$. Transporting the Galois action in the case $B = \OO_F$ yields that \[ \Hom_{\Zp \textnormal{-alg}} \left( A , \OO_F \right) \ni f \stackrel{\gamma}{\mapsto} \Big( a \mapsto \gamma (f(a)) \Big). \]
It is thus immediate to see that an element $f \in \Hom_{\Zp \textnormal{-alg}} \left( A , B \right)$ is Galois-invariant if and only if it has image contained in $\Zp$, which is to say it factors through $A \lra \Zp \hookrightarrow \OO_F$.

Let $\widetilde A$ be the $\OO_F$-algebra of the affine group $\mathrm T$. The universal property of Weil restriction says that for any $\Zp$-algebra $C$ we have a natural adjunction isomorphism \[ \Hom_{\OO_F \textnormal{-alg}} \left( \widetilde A , C \otimes_{\Zp} \OO_F \right) \cong \Hom_{\Zp \textnormal{-alg}} \left( A , C \right), \] functorial in $C$.
In particular, taking first $C = \Zp$ and then $C= \OO_F$ gives the following diagram, where the rows are canonical isomorphisms and the columns are inclusions of Galois-invariants, induced by the obvious post-compositions on $\Hom$-groups: \begin{displaymath} \xymatrix{ \mathrm T(\OO_F) = \Hom_{\OO_F \textnormal{-alg}} \left( \widetilde A , \Zp \otimes_{\Zp} \OO_F \right) \ar[r]^{\sim} \ar@{^{(}->}[d] & \Hom_{\Zp \textnormal{-alg}} \left( A , \Zp \right) = \Res{\OO_F}{\Zp} \mathrm T(\Zp) \ar@{^{(}->}[d] \\ \mathrm T( \OO_F \otimes_{\Zp} \OO_F ) = \Hom_{\OO_F \textnormal{-alg}} \left( \widetilde A , \OO_F \otimes_{\Zp} \OO_F \right) \ar[r]^{\sim} & \Hom_{\Zp \textnormal{-alg}} \left( A , \OO_F \right) = \Res{\OO_F}{\Zp} \mathrm T (\OO_F) } \end{displaymath}
Finally, the isomorphism of $\OO_F$-algebras \[ \OO_F \otimes_{\Zp} \OO_F \lra \prod_{\sigma} \OO_F \otimes_{\OO_F, \sigma} \OO_F \qquad x \otimes y \mapsto \left( x \otimes_{\sigma} y \right)_{\sigma} \] shows that the subalgebra $\Zp \otimes_{\Zp} \OO_F$ maps into the `diagonally embedded' $\{ \left( 1 \otimes_{\sigma} x_{\sigma} \right)_{\sigma} \, | \, x_{\sigma} = x_{\sigma'} \textnormal{ for all } \sigma, \sigma' \}$.

Therefore, along the isomorphisms \[ \mathrm T \left( \OO_F \otimes_{\Zp} \OO_F \right) \cong \mathrm T \left( \prod_{\sigma} \OO_F \otimes_{\OO_F, \sigma } \OO_F \right) \cong \prod_{\sigma} \mathrm T \left( \OO_F \otimes_{\OO_F, \sigma} \OO_F \right) \cong \prod_{\sigma} \left( \mathrm T \times_{\sigma^{-1}} \OO_F \right) (\OO_F) \] the Galois-fixed torus $\Res{\OO_F}{\Zp} \mathrm T (\Zp) = \mathrm T(\OO_F)$ embeds `diagonally'.

In particular, equation \ref{extendedtorusaction} shows that an element of this diagonally embedded torus $t = (t_{\sigma})_{\sigma}$ (with all $t_{\sigma}$ being equal, and henceforth denoted by $t$) acts on $\gamma. f_{\underline w}$ as multiplication by $\prod_{\sigma} \left( w_{\sigma} \cdot 0 \right) (t)$. This is independent of $\gamma$, and therefore the $\mathrm T(\OO_F)$ action on $e .f_{\underline w}$ is via the same character.

We have thus shown that every irreducible $\Res{\OO_F}{\Zp} \mathrm T(\Zp) = \mathrm T(\OO_F)$-submodule of $H^n_{\Zp} \left( \mathfrak n, \Zp \right)$ is of the form $V_{\OO_F} \left( \sum_i w_i \cdot 0 \right)$ for $w_1, \ldots, w_d \in W$ with $\sum_i l(w_i)=n$. It remains to find out the multiplicity of each of these irreducible modules.

Fix then an unordered $d$-uple $\underline w = \{ w_i \}_{i=1}^d \in W^d$ with $\sum_i l(w_i)=n$, and possibly some $w_i$'s are repeated.
The bijections $\sigma: \Aut (F / \Qp) \lra \underline w$ index the cohomology classes $f_{\sigma.\underline w} = \otimes_{\sigma \in \Aut(F/\Qp)} f_{w_{\sigma}}$ such that $e(f_{\sigma.\underline w})$ is in the $\sum_i \left( w_i \cdot 0 \right)$-eigenspace for the Galois-fixed subtorus $\mathrm T(\OO_F)$, and in fact we have shown that the the images under the projection operator $e$ of the various $f_{\underline w}$ (as we vary $\underline w$) yield a basis of $H^*_{\Zp} \left( \mathfrak n , \Zp \right)$.

Two different bijections $\sigma, \sigma': \Aut (F / \Qp) \lra \underline w$ yield cohomology classes $f_{\sigma.\underline w}$ and $f_{\sigma' . \underline w}$ who have the same image under the projection operator $e$ if and only if there is some $\gamma \in \Gal(F / \Qp)$ such that $\gamma.f_{\sigma.\underline w} = f_{\sigma' . \underline w} $, which is to say if and only if $\sigma. \underline w$ and $\sigma'. \underline w$ are in fact in the same Galois orbit. This shows the formula for the multiplicity.
%
%

Finally, we notice that we can relax the assumption that $p$ does not divide $d$: indeed, we consider $H^*_{\OO_F} \left( \mathfrak n_{\OO_F}, \OO_F \right) \otimes_{\OO_F} F$, with the Galois and the torus action only on the first factor. The averaging operator $e$ is well defined on this tensor product (even if it does not preserve the lattice $H^*_{\OO_F} \left( \mathfrak n_{\OO_F}, \OO_F \right)$), and is a projection operator on $\left( H^*_{\OO_F} \left( \mathfrak n_{\OO_F}, \OO_F \right) \otimes_{\OO_F} F \right)^{\Gal(F/\Qp)} \cong H^*_{\OO_F} \left( \mathfrak n_{\OO_F}, \OO_F \right)^{\Gal (F / \Qp)} \otimes_{\OO_F} F$.
The same argument as before shows now that $\mathrm T(\OO_F)$ acts on $H^n_{\OO_F} \left( \mathfrak n_{\OO_F}, \OO_F \right)^{\Gal (F / \Qp)} \otimes_{\OO_F} F$ via the characters $V_{\OO_F} \left( \sum_{\sigma} w_{\sigma} \cdot 0 \right) \otimes_{\OO_F} F$ as we vary $\underline w = (w_{\sigma})$ among Weyl group elements whose sums of lengths is $n$.

We state the result.
\kostantcorollary*

\section{Comparison between group and Lie algebra cohomology} \label{finalsection}
In this section we prove theorem \ref{groupLiecomparison}, whose outline of the proof has been given in the introduction.
As mentioned in subsection \ref{reductionsteps}, we can and do assume that $\mathrm G$ is simple.

\subsection{The Lie algebra as a graded algebra}
In this subsection we show that $\mathfrak n$ is naturally isomorphic to the graded Lie algebra associated to the lower central series filtration of $N$.
We are inspired by two sources that prove very similar results. One of the source is Polo-Tilouine \cite{PT}, sections 3.3 and 3.6, and the other is Friedlander-Parshall \cite{FP}, proposition 4.2.
Both sources are considering nilpotent groups over $\Z$, but adapting it to our $p$-adic situation is not difficult.

Let $\{ C^i(N)\}_{i \ge 1}$ be the lower central series of $N$ - this has a description in terms of root groups as in lemma \ref{centralseriesandrootgroups} below.
%

Consider the graded algebra $\gr N : = \bigoplus_i C^i(N) / C^{i+1}(N)$. This is a $\Zp$-Lie algebra (with brackets given by taking commutators), since each quotient $C^i(N) / C^{i+1}(N)$ is a finite, free $\Zp$-module - again by lemma \ref{centralseriesandrootgroups}.

Our running assumption on the residue characteristic $p \ge 5$ implies that the coefficients appearing in the Chevalley commutation relations (see for instance Carter \cite{carter}, theorem 5.2.2) are invertible in $\Zp$, which is crucial for the following result.
\begin{lem} \label{centralseriesandrootgroups} For all $n \ge 1$ we have \[ C^n(N) = \prod_{h(\alpha) \ge n} N_{\alpha} \] where as usual the product on the right is taken with respect to our fixed non-decreasing ordering of the roots.
\end{lem}
\begin{proof}
By induction on $n$, the case $n=1$ being given by the isomorphism in equation \ref{multiplicationisom}.

Suppose the claim holds for $n-1$, so that $C^{n-1}(N) = \prod_{h(\alpha) \ge n-1} N_{\alpha}$. By definition, $C^n(N) = [N, C^{n-1}(N)]$ so the commutation relations in \cite{carter} (theorem 5.2.2) make it clear that the inclusion $C^n(N) \subset \prod_{h(\alpha) \ge n} N_{\alpha}$ holds.

For the other inclusion, we show that the $N_{\beta}$'s having $h(\beta)= m \ge n$ are contained in $C^n(N)$ by decreasing induction on $m$.
For the base case, fix $\beta$ a root of maximal height $h(\beta) = m$ and take a simple root $\alpha_i$ such that $\beta - \alpha_i =\alpha \in \Phi^+$. Then $h(\alpha) = m-1 \ge n-1$, so that $N_{\alpha} \subset C^{n-1}(N)$.
Consider then the commutator \[ [\theta_{\alpha_i}(1), \theta_{\alpha}(1) ] = \theta_{\beta} (c) \] for some $c \in \OO_F^*$ by the aforementioned commutation relations and our assumption on the residue characteristic $p \ge 5$.
The fact that there is no other factor on the right hand side is precisely due to the maximality of $m$, so that there are no other roots of the form $h \alpha_i + k \alpha$ for $i,k \ge 1$.
This shows that $N_{\beta} \subset [C^1(N), C^{n-1}(N)] = C^n(N)$.

Suppose now that we have shown that all $N_{\beta'}$'s having $l(\beta') \ge m+1 > n$ are contained in $C^n(N)$.
Fix $\beta$ with $l(\beta)=m \ge n$, take again a simple root $\alpha_i$ such that $\beta - \alpha_i = \alpha \in \Phi^+$ and consider the commutator \[ [\theta_{\alpha_i}(1), \theta_{\alpha}(1)] = \prod_{h, k \ge 1} \theta_{h \alpha_i + k \alpha} \left( c_{h,k,\alpha_i,\alpha} \cdot (\pm 1) \right). \]
The only positive root on the right hand side whose height is less than $m+1$ is the one corresponding to $h=1=k$, that is, it is exactly $\beta$.
Therefore, \[ [\theta_{\alpha_i}(1), \theta_{\alpha}(1)] = \theta_{\beta} (c) \cdot u' \] for some $c \in \OO_F^*$ and $u' \in \prod_{h(\beta') \ge m+1} N_{\beta'} \subset C^n(N)$. In particular, we obtain that $N_{\beta} \subset C^n(N)$.
\end{proof}

Notice now that the height function $h$ provides us also with a grading (in fact, with a filtration) on the Lie algebra $\mathfrak n$: that is to say, $\mathfrak n (m) = \bigoplus_{h(\alpha) = m} \mathfrak n_{\alpha}$. The usual commutation formulas show that the bracket operation is compatible with the grading, i.e. $\left[ \mathfrak n(m), \mathfrak n(n) \right] \subset \mathfrak n(m+n)$.

\begin{thm} \label{unipotentLieiso} The family of maps \[ \theta_n: C^n(N) / C^{n+1}(N) \cong \prod_{h(\alpha)=n} N_{\alpha} \stackrel{\prod \theta_{\alpha}^{-1}}{\lra} \bigoplus_{h(\alpha)=n} \mathfrak n_{\alpha} \]
indexed by $n$ provides an isomorphism of graded $\Zp$-Lie algebras $\gr N \cong \mathfrak n$.
\end{thm}
\begin{rem} We are making a small abuse of notation by implicitly identifying $\mathbb G_a \cong \Lie \mathrm U_a \cong \mathfrak u_{\alpha}$ via $\operatorname d \theta_{\alpha}$.
\end{rem}
\begin{proof}
The fact that the maps $\theta_n$ are $\Zp$-linear and provide a bijection $\gr N \lra \mathfrak n$ is clear from lemma \ref{centralseriesandrootgroups}, what needs to be shown is that the family of maps $\{ \theta_n \}$ respects the brackets. This rests on the usual commutation formulas of \cite{carter}, section 5.2, as explained by Friedlander and Parshall in their proof of proposition 4.2 in \cite{FP}.
\end{proof}

\subsection{The universal enveloping algebra}
We now prove that the filtration given by the powers of the augmentation ideal on the completed group algebra $\Zp[[N]]$ has associated graded algebra isomorphic to the universal enveloping algebra $\mathcal U(\mathfrak n)$. We adapt the argument from Hartley's exposition in \cite{hartley1}, section 2.

Let then $I = \langle x - 1 \rangle_{x \in N}$ be the augmentation ideal, that is to say, the kernel of the augmentation map $\Zp[[N]] \lra \Zp$. This is a two sided ideal and is closed for the profinite topology of $\Zp[[N]]$.



We filter $\Zp[[N]]$ by powers of $I$, and consider the associated graded algebra \[ \gr \Zp[[N]] : = \bigoplus_{n \ge 0}  I^n /  I^{n+1}. \]
The associativity of the multiplication operation in the group algebra yields that the operation \[ (x +   I^{n+1})(y +  I^{m+1}) :=xy +  I^{m+n+1} \] on $\gr \Zp[[N]]$ is associative. In particular, we get a graded $\Zp$-Lie algebra structure, where the bracket operation is defined as $[x,y] = xy -yx$.

It is well-known (see for example Passi \cite{passi} chapter III section 1.5) that if $x \in C^n(N)$ (the $n$-th subgroup in the lower central series for $N$) then $x-1 \in J^n$, the $n$-th power of the augmentation ideal $J \subset \Zp[N]$, the (uncompleted) group algebra. In particular, the closure of $J$ in $\Zp[[N]]$ is exactly $I$, and the same holds for the respective powers. 

We have then a morphism of graded $\Zp$-Lie algebras: \begin{equation} \label{inducingmap} \gr N \lra \gr \Zp[[N]] \qquad x C^{n+1}(N) \mapsto (x-1) + I^{n+1} \quad \forall x \in C^n(N). \end{equation}
Indeed it is easy to check that \[ xyx^{-1}y^{-1} -1 = [x,y]_N -1 \equiv [x-1,y-1]_{\gr \Zp[[N]]} = (x-1)(y-1)-(y-1)(x-1) = xy-yx \bmod I^{m+n+1} \] for any $x \in C^n(N), y \in C^m(N)$.

We denote by $\mathcal U(\gr N)$ the universal enveloping algebra of $\gr N \cong \mathfrak n$ as a graded $\Zp$-Lie algebra. That is, $\mathcal U(\gr N)$ is the initial object in the category of $\Zp$-associative graded Lie algebras equipped with graded $\Zp$-Lie algebra maps from $\gr N$.
\begin{rem}
In fact, $\mathcal U(\gr N)$ is the usual enveloping algebra of the $\Zp$-Lie algebra $\gr N$, with the grading induced by that of $\gr N$.
\end{rem}

Since $\gr \Zp[[N]]$ is a graded associative $\Zp$-Lie algebra, the above map \ref{inducingmap} extends to \begin{equation} \label{envelopinmap} \Theta: \mathcal U(\gr N) \lra \gr \Zp[[N]]. \end{equation}
We will show in theorem \ref{envelopingalgebraiso} that $\Theta$ is an isomorphism.

We now describe a homogeneous $\Zp$-basis for the graded Lie algebra $\gr N$.
Recall from lemma \ref{centralseriesandrootgroups} that $C^n(N) = \prod_{h(\alpha) \ge n} N_{\alpha}$, so that we have an isomorphism $\prod_{h(\alpha) = n} N_{\alpha} \cong C^n(N) / C^{n+1}(N)$ induced by the injection of the root groups inside $C^n(N)$.

Fix once and for all an isomorphism $\OO_F \cong \Zp^d$ of free $\Zp$-modules. Then we have $\Zp^d \cong \OO_F \cong N_{\alpha}$ where the second isomorphism is $\theta_{\alpha}$. Since both maps are $\Zp$-linear, this provides us with a $\Zp$-basis $B_{\alpha} = \{ x^{\alpha}_1, \ldots, x^{\alpha}_d \}$ (the image of the standard basis of the free $\Zp$-module $\Zp^d$) for each $N_{\alpha}$.

We conclude that \[ B_n = \bigcup_{h(\alpha) = n} B_{\alpha}\] (as usual, ordered accordingly to the ordering on the roots fixed before) is a $\Zp$-basis for $C^n(N) / C^{n+1}(N)$.
The union of the $B_n$'s as $n$ increases provides us then with a $\Zp$-basis $B = \{ x_1, \ldots, x_M \}$ of $\gr N$ formed by homogeneous elements and ordered in (weakly) increasing degree. 
For each element $x_j \in B$ we denote by $\mu(x_j)$ the integer such that $x_j \in B_n$ - that is to say, the height of the root $\alpha(j)$ such that $x_j \in N_{\alpha(j)}$ - and by $X_j$ the subgroup of $N_{\alpha(j)} \cong \Zp^d$ which is `$\Zp$-spanned' by $x_j$ (an isomorphic image of $\Zp$).

\begin{lem} \label{augmentationlemma}
For each $n \ge 1$, define $E_n$ to be the $\Zp$-span (inside $\Zp[N]$) of the elements \[ (1 - y_1) \cdot \ldots \cdot (1 - y_s) \textnormal{ such that } y_j \in X_{k_j} \textnormal{ and } \sum \mu(x_{k_j}) \ge n. \]
Here, we ordered the products so that if $j < j'$ then $k_j \le k_{j'}$ which is to say that $x_{k_j}$ comes no later than $x_{k_{j'}}$ in our ordering fixed above.\footnote{We could have $x_{k_j} = x_{k_{j'}}$.}
Then $J^n = E_n$, and in particular the closure of $E_n$ in $\Zp[[N]]$ coincides with $I^n$.
\end{lem}
\begin{proof}
This follows closely section 2 in \cite{hartley2}.
First of all just like in lemma 2.2 of loc. cit., one shows that $E_n E_m \subset E_{n+m}$ - the proof goes through verbatim, since it's not much more than a re-arrangement lemma. In particular since the product map in $\Zp[[N]]$ is continuous we can take closures to get \[ \overline {E_a} \cdot \overline {E_b} \subset \overline {E_{a+b}} \textnormal{ for all } a,b \ge 1. \]

Now notice that clearly $E_1 \subset J$. On the other hand, each $g \in N$ can be written down uniquely (thanks to the isomorphism in formula \ref{multiplicationisom}) as an ordered product $g = \prod_{\alpha} g^{\alpha}$ with $g^{\alpha} \in N_{\alpha}$.
Since we fixed isomorphisms $N_{\alpha} \cong \Zp^d$, we obtain that $g^{\alpha} = \prod_{i=1}^d y_i^{\alpha}$ (again an ordered product) for $y_i^{\alpha} \in X_i^{\alpha}$.
Replacing each $y_i^{\alpha}$ by $\left( 1 - (1 - y_i^{\alpha}) \right)$ and expanding the product one gets that \[ g = 1 + \sum_{\textnormal{all products}} (-1)^s (1-y'_1) \cdot \ldots \cdot (1-y'_s). \]
In particular, $1-g$ is in $E_1$. Since $\{ 1-g \}_{g \in N}$ is a $\Zp$-basis for the augmentation ideal $J \subset \Zp[N]$, we obtain that $E_1 = J$, and taking closures that $\overline E_1 = I$.

Now, as $J^{n+1} = J^n \cdot J$, we obtain by induction on $n$ that $J^n \subset E_n$.

The other inclusion is clear, since $y_j \in \Zp.x_j \subset C^{\mu(x_j)}(N)$ implies that $(1-y_j) \in J^{\mu(x_j)}$. Therefore, $J^n = E_n$ and taking closures, we obtain \[ \overline {E_n} = \overline {J^n} = I^n. \]
\end{proof}
%
For each element $x_k$ of our basis $B$ and each $r \ge 0$, we define the following elements of $\Zp[N]$: \[ u_r^{(k)} = x_k^{- \ceil{\frac{r}{2}}} (x_k-1)^r = \left\{ \begin{array}{ccc} 1 & \textnormal{ if } r=0 \\ x_k^{-i} (x_k-1)^{2i} & \textnormal{ if } r=2i \\ x_k^{-i}(x_k-1)^{2i-1} & \textnormal{ if } r=2i-1 \end{array} \right. \]
For each $M$-uple of non-negative integers $\mathbf j = (j_1, \ldots, j_M)$ we set \[ u(\mathbf j) := u_{j_1}^{(1)} \cdot \ldots \cdot u_{j_M}^{(M)} \] having weight \[ \mu(\mathbf j) := \sum_{k=1}^M j_k \mu(x_k). \]
\begin{prop} \label{augmentationpowersbases}
A topological $\Zp$-basis of $I^n$ is given by \[ \mathcal B_n = \left\{ u(\mathbf j) \, | \, \mu(\mathbf j) \ge n \right\}. \]
By `topological $\Zp$-basis' we mean that these elements are $\Zp$-linearly independent, and that the closure of their $\Zp$-span in $\Zp[[N]]$ coincide with $I^n$.
\end{prop}
\begin{proof}
Denote $I_k \subset \Zp[[X_k]]$ the augmentation ideal. Notice that if $y_j \in X_{k_j}$, then $(1-y_j) \in I_{k_j}$.
Grouping together the factors $(1-y_j)$'s corresponding to the same subgroup $X_k$, we obtain that $E_n$ is spanned by products $\xi_1 \cdot \ldots \xi_M$ such that $\xi_k \in I_k^{j_k}$ and $\sum j_k \mu(x_k) \ge n$ (we allow $j_k=0$ if no factor corresponding to $X_k$ appears).

We now construct topological $\Zp$-bases for the powers of the augmentation ideal $I_k$. In fact, we already have:
\begin{claim} \label{Zpgroupalgebra}
The set $\{ u_t^{(k)} \}_{t \ge r}$ is a topological $\Zp$-basis for $I_k^r$ - i.e. they are $\Zp$-linearly independent and their $\Zp$-span is dense in $I_k^r$.
\end{claim}
\begin{proof}[Proof of the claim] \setlength{\leftskip}{1cm} Exactly like in lemma 2.4 in \cite{hartley2} one proves that $\{ u_0^{(k)}, \ldots, u_r^{(k)} \}$ is a $\Zp$-basis for $\sum_{i=- \ceil{\frac{r}{2}}}^{\floor{\frac{r}{2}}} \Zp x_k^i$.
In particular, $\{ u_r^{(k)} \}_{r \ge 0}$ is a $\Zp$-basis for $\bigoplus_{i \in \Z} \Zp x_k^i$, which shows the $\Zp$-linear independence statement.
Now notice that once we take the closure of this in $\Zp[[X_k]]$, we obtain the entire group algebra.
This is best seen from the fact that $\Z \hookrightarrow \Zp$ is a dense embedding, or even more explicitly from the well-known isomorphisms $\psi: \Zp[[X_k]] \lra \Zp[[T]]$ (ring of formal power series in one variable) obtained by sending $(x_k -1) \mapsto T$.

The definition of the $u_i^{(k)}$'s makes it clear that $I_k^r \supset \sum_{t \ge r} \Zp u_t^{(k)}$.
On the other hand, we have \begin{equation} \label{directsumgroupalgebra} \Zp[[X_k]] = \sum_{i=- \ceil{\frac{r}{2}}} ^{\floor{\frac{r}{2}}} \Zp x_k^i \oplus \overline{\sum_{i \ge r} \Zp u_i^{(k)}}, \end{equation} again this can be seen by considering the analogous statement in the isomorphic ring $\Zp[[T]]$.

Intersecting both sides of equation \ref{directsumgroupalgebra} with $I_k^r$ yields that $I_k^r = \left( I_k^r \cap \sum_{i=- \ceil{\frac{r}{2}}} ^{\floor{\frac{r}{2}}} \Zp.x_k \right) \oplus \overline{\sum_{i \ge r} \Zp u_i^{(k)}}$, so it suffices to prove that the first summand is zero to complete the proof of the claim.

We have shown that $\sum_{i=- \ceil{\frac{r}{2}}} ^{\floor{\frac{r}{2}}} \Zp x_k^i$ is the $\Zp$-span of $\{ u_0^{(k)}, \ldots, u_{r-1}^{(k)} \}$.
Applying the aforementioned isomorphism $\Zp[[X_j]] \cong \Zp[[T]]$ and working in the target ring, one sees that $I_j^r = (T^r)$ but the only element in the $\Zp$-span of $\psi \left( \{ u_0^{(k)}, \ldots, u_{r-1}^{(k)} \} \right) = \{ 1, T(T+1)^{-1}, T^2(T+1)^{-1}, \ldots, T^{r-1} (T+1)^{- \ceil{\frac{r}{2}}} \} $ divisible by $T^r$ is zero.
\end{proof}
\setlength{\leftskip}{0pt}

The claim showed that the set $\{ u_r^{(k)} \}_{r \ge j_k}$ is a topological $\Zp$-basis for $I_k^{j_k}$, or in other words that $\overline {\{ u_r^{(k)} \}_{r \ge j_k} } = I_k^{j_k}$. Therefore the generic element spanning $E_n$ is of the type \[ \xi_1 \cdot \ldots \cdot \xi_M \textnormal{ with } \xi_k \in I_k^{j_k} = \overline {\{ u_r^{(k)} \}_{r \ge j_k} } \textnormal{ and } \sum_k j_k \mu(x_k) \ge n. \]
But now continuity of the product map in $\Zp[[N]]$ means that the product of the closures $\prod_k \overline {\{ u_r^{(k)} \}_{r \ge j_k} }$ is contained in the closure of the product $ \prod_k \{ u_r^{(k)} \}_{r \ge j_k}$ and this latter product coincide exactly with the $\Zp$-span of the $u(\mathbf j)$ having $\mu(\mathbf j) \ge n$.

Finally, the $\Zp$-linear independence of the $u(\mathbf j)$'s follows from the $\Zp$-linear independence of the single factors $u_{j_k}^{(k)}$ proved in claim \ref{Zpgroupalgebra}.
\end{proof}
We are finally ready to show the main theorem of this subsection - the bulk of the work has already been done.
\begin{thm} \label{envelopingalgebraiso} The map $\Theta: \mathcal U(\gr N) \lra \gr \Zp[[N]]$ defined as in formula \ref{envelopinmap} is an isomorphism of graded associative $\Zp$-Lie algebras.
\end{thm}
\begin{proof} Recall that $\Theta$ is induced by the map $\gr N \lra \gr \Zp[[N]]$ sending $x C^{n+1}(N) \mapsto (x-1) + I^{n+1}$. We will describe $\Zp$-basis for $\mathcal U(\gr N)$ and $\gr \Zp[[N]]$ and show that $\Theta$ sends one into the other.

The ordered basis $B$ for $\gr N$ provides, by the Poincar\'e-Birkhoff-Witt theorem, a $\Zp$-basis for $\mathcal U(\gr N)$ consisting of the monomials \[ x_1^{j_1} \cdot \ldots \cdot x_M^{j_M} \] for all $M$-uples of non-negative integers $\mathbf j = (j_1, \ldots j_m)$.
Moreover, this is a basis of homogeneous elements for the grading on $\mathcal U(\gr N)$ induced by the grading on $\gr N$, where \[ \deg \left( x_1^{j_1} \cdot \ldots \cdot x_M^{j_M} \right) = \sum_{k=1}^M j_k \mu(x_k). \]
In particular, a basis for the subspace of degree $n$ homogeneous elements of $\mathcal U(\gr N)$ is given by \begin{equation} \left\{ x_1^{j_1} \cdot \ldots \cdot x_M^{j_M} \right\} \textnormal{ such that } \sum_k j_k \mu(x_k) = n. \label{envelopingbasis}
\end{equation}

We switch our attention to $\gr \Zp[[N]]$. For each $M$-uple of non-negative integers $\mathbf j = (j_1, \ldots, j_m)$, we notice that \[ u(\mathbf j) \equiv v(\mathbf j) := (x_1-1)^{j_1} \cdot \ldots \cdot (x_M-1)^{j_m} \bmod I^{\mu(\mathbf j) +1} \] since when we compute $u(\mathbf j) - v(\mathbf j)$ we pick up an additional factor of $(x_k^{\pm 1} - 1)$, which is in $I$.

Therefore, a $\Zp$-basis\footnote{We can drop 'topological' since $I^n / I^{n+1}$ is a finite, free $\Zp$-module and thus we do not need to take the closure of a $\Zp$-span.} of $I^n / I^{n+1}$ is given by \begin{equation} \label{groupalgebrabasis} \{ v(\mathbf j) + I^{n+1} \} \textnormal{ such that } \mu(\mathbf j) = n. \end{equation}

Finally, we conclude by noticing that the map $\Theta$ gives \[ x_1^{j_1} \cdot \ldots \cdot x_M^{j_M} \mapsto (x_1-1)^{j_1} \cdot \ldots \cdot (x_M-1)^{j_M} = v(\mathbf j) \] so that the two $\Zp$-basis correspond to each other.
\end{proof}

\subsection{Comparison via a spectral sequence} \label{spectralsection}
In this subsection we introduce a spectral sequence which relates cohomology of (modules over) a filtered algebra and cohomology of (modules over) the associated graded algebra. This is a crucial ingredient to prove theorem \ref{groupLiecomparison}.

We start by recalling a general setup for filtered and graded rings and modules.

We follow Polo and Tilouine's ideas from \cite{PT}, but we will need take into account the topology of our profinite algebras (like $\Zp[[N]]$) as well. In particular, this calls for adapting the results of Grunenfelder (\cite{grunenfelder}) and Sjodin (\cite{sjodin}) to the topological setup.

Let then $A$ be a profinite topological ring (not necessarily commutative) equipped with a descending filtration $\{ F^n A\}_{n \ge 0}$ such that each $F^n A$ is a closed two-sided ideal in $A$ (for example, the filtration induced by the powers of a closed two-sided ideal). Every $A$-module is assumed to be a left $A$-module, every left $A$-module is assumed to have a topology for which the action map is continuous, and every morphism of left $A$-modules is assumed to be continuous.

We consider the category $\mathcal F_A$ of filtered (topological) left $A$-modules. All filtrations are decreasing and indexed by $\N$ (unless we say otherwise, see for example the filtration on $\Hom$-spaces below), and each $F^n M$ is closed in $M$. The action of $A$ is such that $F^h A . F^k M \subset F^{h+k} M$.

\begin{defn} Given a map of filtered $A$-modules $f:M \lra N$ we say that $f$ has filt-degree $k$ if $f(F^h M) \subset F^{h+k} N$ for all $h$ (where we assume $F^t N = N$ if the filtration on $N$ is indexed by $\N$ and $t < 0$).
\end{defn}
\begin{rem} Notice that the filt-degree of a morphism is not unique, if it exists: indeed if $f$ has filt-degree $k$ than it also has filt-degree $k-i$ for each $i \ge 1$, because the filtration on $N$ is descending.
\end{rem}
We denote by $\underline \Hom (M,N)$ the morphism of finite filt-degree, but as morphisms in the category $\mathcal F_A$ we only consider the morphisms of filt-degree $0$.

\begin{defn} A filt-degree $0$ morphism $f \in \Hom_{\mathcal F_A}(M,N)$ is said to be strict if $f(F^h M)  = \im f \cap F^h N$ for all $h$.
\end{defn}

\begin{defn} The filtration $\left\{ F^n M \right\}$ of an $A$-module $M$ is said to be \begin{itemize}
\item \emph{discrete} if $F^n M = 0$ for $n$ large enough.
\item \emph{complete} if $M = \varprojlim_n M / F^n M$.
\item \emph{exhaustive} if $M = \bigcup_n F^n M$. This is usually the case, if we set e.g. $F^0 M = M$ and $F^k M = F^k A . M$.
\item \emph{bounded} if $F^n M = M$ for some $n \in \Z$.
\end{itemize}
\end{defn}

We also have the corresponding concepts of graded topological ring, graded topological left modules, and so on. In particular if $A$ is a filtered topological ring, then $\gr A$ is a graded topological ring.
The mapping $M \mapsto \gr M$ defines a functor from $\mathcal F_A$ to $\mathcal G_{\gr A}$, the category of graded left $\gr A$-modules. If $B$ is a graded ring, the morphisms in the category $\mathcal G_B$ preserve the degree.

\begin{fact}[\cite{sjodin}, section 0]
A short exact sequence of strict morphisms in $\mathcal F_A$ stays exact upon applying $\gr$.
\end{fact}

Notice that for each $M, N \in \mathcal F_A$ we can $\Z$-filter $\underline \Hom (M,N)$ by \[ F^n \underline \Hom (M,N) = \left\{ f: M \lra N \textnormal{ of filt-degree } n \right\}. \]
\begin{rem} \begin{enumerate}
    \item This filtration is exhaustive, since we defined $\underline \Hom (M,N)$ to be the morphisms respecting the filtration and having finite filt-degree (possibly negative).
    \item It is often the case that $\underline \Hom (M,N) \neq \Hom_A(M,N)$, i.e. that there are $A$-module maps $M \lra N$ which do not have a filt-degree (see section 4 of \cite{sjodin} for a counterexample).
\end{enumerate}
\end{rem}
We now want conditions to ensure that each $A$-module map $f: M \lra N$ has a finite filt-degree.

\begin{defn} \label{filtfreeandfg} Let $M \in \mathcal F_A$. \begin{itemize}
\item We say that $M$ is filt-free if it is a free object in $\Mod(A)$, and there is a basis $\{x_i\}$ with integers $h(i)$ such that \[ F^n M = \sum_i F^{n-h(i)}A . x_i. \qquad \forall n. \]
Equivalently, $M$ is a direct sum of shifts\footnote{Explicitly, for a filtered $A$-module $M$, the shift $M^{(h)}$ is the filtered $A$-module defined by $F^n M^{(h)} = F^{n-h}M$.} of $A$, as mentioned in section 3.4 of \cite{PT}.
\item We say that $M$ is filt finitely generated (filt f.g.) if there exists a finite set $\{x_i, h(i)\}$ with $x_i \in M$ and $h(i) \in \Z$ such that \[ F^n M = \sum_i F^{n-h(i)} A. x_i \qquad \forall n. \]
\end{itemize}
\end{defn}

\begin{lem}[Lemma 15 in \cite{sjodin}] Let $M,N \in \mathcal F_A$. Suppose that $M$ is filt f.g. and that the filtration on $N$ is exhaustive. Then \[ \Hom_A(M,N) = \underline \Hom (M,N). \]
\end{lem}

We also have a map \[ \phi(M,N): \gr \underline \Hom (M,N) \lra \Hom_{\gr A} \left( \gr M, \gr N \right) \] defined in the obvious way, and functorial in $M$ and $N$. 

\begin{lem}[Lemma 16 in \cite{sjodin}] \label{grHomcommute}
Suppose $M$ is filt-free. Then $\phi(M,N)$ is an isomorphism for all $N \in \mathcal F_A$.
\end{lem}
%
%
\begin{prop}[Lemmas 1 and 2 in \cite{sjodin} and remark 5.2.2 in \cite{NVO}, see also proposition (C) in \cite{PT}] \label{strictresolutionexists}
Every $M \in \mathcal F_A$ admits a filt-free resolution $L_{\bullet} \lra M \lra 0$ by strict morphisms (called a \emph{strict resolution}).

If the filtration on $A$ is exhaustive and complete and the associated graded $\gr A$ is Noetherian, then every $M \in \mathcal F_A$ such that $\gr M$ is finitely generated admits a strict resolution by filt-free, f.g. modules $L_i$.

If $M$ is filt f.g. and has an exhaustive filtration then we can also choose a strict resolution by filt-free, f.g. modules $L_i$.
\end{prop}
\begin{proof}
The first two statements are taken verbatim from the aforementioned sources. Remark 5.2.2 says that if $M \in \mathcal F_A$ is filt f.g. and has an exhaustive filtration, then it admits a strict epimorphism $f_0: L \lra M \lra 0$ with $L$ filt-free and finitely generated. Replacing $M$ by $\ker f_0$ with the induced filtration as a submodule of $L$ (see section I.2 in \cite{NVO}) and iterating yields the required filt-free, f.g. resolution.
\end{proof}
Suppose now that $M \in \mathcal F_A$ is filt f.g. and that either of the latter two statements of the proposition applies, so that $M$ admits a strict, filt-free, f.g. resolution: \[ \ldots \lra L_n \stackrel{f_n}{\lra} \ldots \lra L_1 \stackrel{f_1}{\lra} L_0 \stackrel{f_0}{\lra} M \lra 0, \] recall in particular that each $f_i$ has filt-degree $0$.

Fix $N \in \mathcal F_A$ and suppose its filtration is exhaustive. Then applying $\Hom_A ( - , N )$ to the above resolution yields the complex \[ 0 \lra \Hom_A(M,N) \lra \Hom_A (L_0, N) \stackrel{d_0}{\lra} \Hom_A (L_1, N) \stackrel{d_1}{\lra} \ldots \lra \Hom_A(L_n, N ) \stackrel{d_n}{\lra} \ldots \] whose homology is, by definition, $\Ext^*_A(M,N)$ since each $L_i$ is a free (hence projective) $A$-module.

On the other hand, our previous lemmas and the assumption that each $L_i$ is filt f.g. says that $\Hom_A(L_i,N) = \underline \Hom (L_i, N)$ is filtered.
We obtain then that the complex $\Hom_A (L_{\bullet}, N)$ is a filtered complex, since the differentials $d_i$ are compatible with the filtrations (the $d_i$'s have filt-degree $0$ because the $f_i$'s do).

We can consider then the spectral sequence of this filtered complex: following appendix A of \cite{RZ} we have that the $E_1$-page is the homology of the associated graded complex.
More precisely, applying the functor $\gr$ we obtain an associated graded complex \[ 0 \lra \gr \underline \Hom (L_0, N) \stackrel{\gr d_0}{\lra} \gr \underline \Hom (L_1, N) \stackrel{\gr d_1}{\lra} \ldots \lra \gr \underline \Hom (L_n, N) \stackrel{\gr d_n}{\lra} \ldots \] and then $E_1^{r,s} = H^{r+s} \left( \gr \underline \Hom (L_{\bullet}, N) \right)_r$, that is to say, the $(r+s)$-th homology of the complex of degree-$r$ graded pieces.

The spectral sequence has $E_{\infty}$-page being the homology groups of the complex $\Hom_A(L_{\bullet}, N)$, with grading induced by the filtration of the complex. More precisely, $E_{\infty}^{r,s} = ( H^{r+s} \left( \Hom_A(L_{\bullet}, N) \right)_r )$.

As explained in theorem A.3.1(b) of \cite{RZ}, for the spectral sequence to converge it suffices for the filtration on the complex $\Hom_A (L_{\bullet}, N)$ to be bounded and discrete: for each $n$ we want integers $u(n) < v(n)$ such that $F^{u(n)} \Hom_A (L_n, N) = \Hom_A (L_n, N)$ (bounded) and $F^{v(n)} \Hom_A(L_n, N) = 0$ (discrete).

\begin{prop} \label{Homdiscretefilt}
Suppose that $N$ has a bounded and discrete filtration, and that $L$ is filt-free, direct sum of shifts of bounded degree. Then $\Hom_A(L,N) = \underline \Hom(L,N)$ has bounded and discrete filtration.
\end{prop}
\begin{proof}
The assumption on $L$ means that in definition \ref{filtfreeandfg}, $L$ is a (possible infinite) direct sum of shifts $A^{(h(j))}$ with $h(j)$ bounded above and below. We fix such bounds $h \le h(j) \le H$.

We have then $F^h L = \bigoplus_{j \in J} F^{h-h(j)} A  = \bigoplus_{j \in J} A = L$, so that the filtration on $L$ is bounded. In particular, it is exhaustive and hence lemma 13 in \cite{sjodin} guarantees that $\underline \Hom(L,N)$ has a discrete filtration.

Let now $f \in \Hom_A(L,N)$ and $x \in F^n L = \bigoplus_{j \in J} F^n A^{(h(j))}$, so $x = \sum_j x_j$ is a finite sum with $x_j \in F^n A^{(h(j))} = F^{n-h(j)} A$.
Fix $k$ such that $F^k N = N$, which exists as the filtration on $N$ is bounded. Then we have \[ f(x_j) = x_j f(1_j) \in F^{n-h(j)} A. N = F^{n-h(j)} A . F^k N \subset F^{n-h(j)+k} N \subset F^{n-H+k} N \] where the last inclusion is due to $H \ge h(j)$ and the filtration on $N$ being decreasing.

In particular, $f(x) = \sum_j f(x_j) \in F^{n-H+k} N$ and thus $f$ has filt-degree $k-H$. This shows that the every $A$-morphism has finite filt-degree, so that $\Hom_A(L,N) = \underline \Hom (L,N)$, and that the filtration is bounded.
\end{proof}
%
The proposition clearly applies to each $L = L_i$ obtained above, since those are filt-free, f.g. modules, so they are finite direct sums. In particular, we showed that the spectral sequence is convergent.

Consider now the $E_1$-page of the spectral sequence, having $E_1^{r,s} = H^{r+s} \left( \gr \underline \Hom (L_{\bullet}, N) \right)_r$.
Since $\gr \underline \Hom (L_{\bullet}, N) \cong \Hom_{\gr A} \left( \gr L_{\bullet}, \gr N \right)$ by lemma \ref{grHomcommute}, and $\gr L_n$ is a free $\gr A$-module for each $n$, this graded complex is the one computing the higher derived functors of $\Hom_{\gr A} \left( \gr M, \gr N \right)$.
In other words, $H^{r+s} \left( \Hom_{\gr A} \left( \gr L_{\bullet}, \gr N \right) \right)_r \cong \Ext^{r+s}_{\gr A} \left( \gr M, \gr N \right)_r$.

We proved the following theorem:
\begin{thm} \label{abstractspectral} Let $M \in \mathcal F_A$ be such that it admits a strict resolution by filt-free, f.g. modules $L_{\bullet} \lra M \lra 0$. Let $N \in \mathcal F_A$ have a bounded and discrete filtration.
Then there exists a cohomological spectral sequence with first page \[ E_1^{r,s} = \Ext^{r+s}_{\gr A} \left( \gr M, \gr N \right)_r \Rightarrow E_{\infty}^{r,s} = \Ext_A^{r+s} \left(M, N \right)_r \] converging to the graded module associated to the filtered $\Ext_A^{r+s} \left( M,N \right)$.
\end{thm}

Just like Polo and Tilouine in section 3 of \cite{PT}, we also discuss an equivariant version. From now on we assume that $A$ is a $\Zp$-algebra, for ease of notation.
Suppose that we have a group of automorphisms $\Lambda$ acting on the filtered $\Zp$-algebra $A$, i.e. a subgroup of $\Aut_{\Zp}(A)$ that preserves the filtration.

We can then consider the `smash product' $\Z_p[\Lambda] \otimes_{\Zp} A$.
The ring structure on this smash product is given by \[ (\lambda_1 \otimes a_1) (\lambda_2 \otimes a_2) = \lambda_1 \lambda_2 \otimes a_1 \lambda_1(a_2). \]
We denote the smash product by $\Lambda \sharp A$. Notice that the ring embedding $\phi: A \lra \Lambda \sharp A$ sending $a \mapsto 1 \otimes a$ turns $\Lambda \sharp A$ into a free right $A$-module.
Explicitly, $A^{\oplus \Lambda} \stackrel{\sim}{\lra} \Lambda \sharp A$ via the map $(a_{\lambda}) \mapsto \sum_{\lambda} \lambda \otimes \lambda(a_{\lambda})$.

We can put on $\Lambda \sharp A$ the filtration induced by $\phi$, in the sense that $F^n( \Lambda \sharp A)$ is the two-sided ideal generated by $\phi(F^n A)$.
Restriction of scalars along the embedding $A \lra \Lambda \sharp A$ induces a functor $\mathcal F_{\Lambda \sharp A} \lra \mathcal F_A$, so in particular every filtered $\Lambda \sharp A$-module\footnote{Equivalently, a (filtered) $\Lambda \sharp A$-module is a (filtered) $A$-module $M$ with a $\Lambda$-action (preserving the filtration) such that $\lambda(a.m) = \lambda(a).\lambda(m)$ for all $a \in A$, $\lambda \in \Lambda$ and $m \in M$.} can be considered as a filtered $A$-module.
\begin{rem}
Notice that the actual filtration on $M \in \mathcal F_{\Lambda \sharp A}$ does not change, and hence the abelian group $\gr M$ is unambiguously defined, whether we are considering it as a $\gr A$-module or as a $\gr (\Lambda \sharp A)$-module.
\end{rem}

\begin{prop} \label{inducedidealfiltration}
Suppose that $F^n A = I^n$ for a fixed (two-sided) ideal $I$ of $A$. Let $I'$ be the right ideal of $\Lambda \sharp A$ generated by $I$.
Then $I'$ is a two-sided ideal of $\Lambda \sharp A$, and the filtration induced by $\phi$ is the $I'$-adic filtration.
\end{prop}
\begin{proof}
The first part of the claim is an easy check.
For the second, we need to show that if $I' = (\Lambda \sharp A)I$, then $(I')^n = (\Lambda \sharp A)I^n$. The definition of $I'$ settles the base case of the induction: for the inductive step, one needs to show that if \[ \sum_i t_i \hat t_i \in (I')^{n+1} \textnormal{ with } t_i \in I', \, \hat t_i \in (I')^n, \] then we can write this sum as a linear combination of elements $(\lambda \otimes a) (1 \otimes \tau) \in (\Lambda \sharp A) (1 \otimes I^n)$. Obviously it suffices to treat the case of a single product $t \hat t$, and then use the induction assumption and factor the sums to finish the proof.
\end{proof}
Notice that if $M$ and $N$ are filtered $\Lambda \sharp A$-modules, then the group $\Hom_A(M,N)$ has the structure of a (left) $\Lambda$-module: \[ \lambda.f (m) = \lambda \left( f (\lambda^{-1}m ) \right) \qquad \forall \lambda \in \Lambda, \, \forall f \in \Hom_A(M,N). \]
This action preserves the submodule $\underline \Hom_{\mathcal F_A} \left(M,N \right)$ of morphisms of finite filt-degree, and thus descends to a $\Lambda$-action on $\Hom_{\gr A} \left( \gr M, \gr N \right)$ defined via the same formula.
%

Suppose now that $M$ is a filt f.g. $\Lambda \sharp A$-module with exhaustive filtration. By proposition \ref{strictresolutionexists}, we can find a strict resolution of $M$ by filt-free, f.g. $\Lambda \sharp A$-modules: $L_{\bullet} \lra M \lra 0$, so the maps are $\Lambda$-equivariant. 

Let $N \in \mathcal F_{\Lambda \sharp A}$. Applying $\Hom_A \left( - , N \right)$ to the resolution above gives then the complex \[ 0 \lra \Hom_A(M,N) \lra \Hom_A \left( L_{\bullet}, N \right) \] whose homology computes $\Ext_A^* (M,N)$, and moreover the differentials are now $\Lambda$-equivariant, so that they induce a $\Lambda$-action on homology.

Suppose $N$ has a bounded and discrete filtration as a $\Lambda \sharp A$-module. Then these properties are preserved once we consider $N$ as a filtered $A$-module, hence proposition \ref{Homdiscretefilt} applies (because the $L_i$'s are filt-free, f.g. as $\Lambda \sharp A$-modules, and hence considered as $A$-modules they stay filt-free and the shifts degrees are the same as those as a $\Lambda \sharp A$-module, in particular they are bounded) and our complex \[ 0 \lra \Hom_A \left( L_{\bullet}, N \right) \] is in fact a flltered complex of $\Lambda$-modules.

When we apply $\gr$ we obtain (since the $L_i$'s are filt-free $A$-module and hence lemma \ref{grHomcommute} applies) the complex of $\Lambda$-modules $0 \lra \gr \underline {\Hom} \left( L_{\bullet},  N \right) $ which is isomorphic to $0 \lra \Hom_{\gr A} \left( \gr L_{\bullet}, \gr N \right)$ and hence its homology is exactly $\Ext_{\gr A}^* \left( \gr M, \gr N \right)$.
Therefore, we obtain a filtered complex spectral sequence of $\Lambda$-modules with the same $E_1$ and $E_{\infty}$-pages as in theorem \ref{abstractspectral} but now every object involved has a $\Lambda$-module structure.
We proved the following result.
\begin{cor} \label{equivariantss} Let $M \in \mathcal F_{\Lambda \sharp A}$ be filt f.g. with an exhaustive filtration, and let $N \in \mathcal F_{\Lambda \sharp A}$ have a bounded, discrete filtration.

Then there exists a cohomological spectral sequence of $\Lambda$-modules with first page \[ E_1^{r,s} = \Ext^{r+s}_{\gr A} \left(\gr M, \gr N \right)_r \Rightarrow E_{\infty}^{r,s} = \Ext_A^{r+s} \left( M, N \right)_r \] converging to the graded $\Lambda$-module associated to the filtered $\Ext_A^{r+s} \left(  M,N \right)$.
\end{cor}
\subsection{Proof of theorem \ref{groupLiecomparison}}
Let's now specialize to the situation we care about. Let $G$ be a compact $p$-adic analytic group, and let $A = \Zp [[G]]$ be the completed group algebra. We filter it via the powers of the augmentation ideal $I$.
As explained by Symonds and Weigel in \cite{SW}, the category of topological $\Zp[[G]]$-modules has two important subcategories, that of profinite modules $\mathcal P_A$ and that of discrete modules $\mathcal D_A$.
We filter each $\Zp[[G]]$-module by the powers of the augmentation ideal $I$: \[ F^n  M := I^n M. \]
Notice that every trivial module $M$ gets then the bounded and discrete filtration $F^0 M = M \supset F^1 M = 0$ since obviously $g-1$ sends each $M \ni m \mapsto 0$ for all $g \in G$.

As explained in \cite{SW}, section 3.2, we can consider the $\Ext$-functors: \[ \Ext: \mathcal P_A \times \mathcal D_A \lra \mathcal D_{\Zp} \] and in particular the continuous group cohomology of a discrete module $V$ is computed by \begin{equation} \label{cohomologydef} H^*(G,V) = \Ext^*_{\Zp[[G]]} \left( \Zp, V \right). \end{equation}
A theorem of Lazard (see the discussion in the introduction of \cite{SW} as well as in section 3.7 of loc. cit.) says that compact $p$-adic analytic groups are `of type $\mathbf {FP}_{\infty}$' which is to say that the profinite $\Zp[[G]]$-module $\Zp$ admits a resolution \[ \ldots \lra L_n \lra \ldots \lra L_1 \lra L_0 \lra \Zp \lra 0 \] with finite free $\Zp[[G]]$-modules $L_i$. 

In fact, the discussion in section 3.7 of \cite{SW} (see in particular after proposition 3.7.1 and theorem 3.7.4) says that for a group $G$ of type $\mathbf{FP}_{\infty}$ we can extend the definition of cohomology as in equation \ref{cohomologydef} to coefficient modules in $\mathcal P_A$.
Given $V \in \mathcal P_{\Zp[[G]]}$, its cohomology is defined as the homology of the complex $\Hom_{\Zp[[G]]} (L_{\bullet}, V)$ where $L_{\bullet} \lra \Zp \lra 0$ is any resolution of $\Zp$ by finite, free $\Zp[[G]]$-modules. \\

Let us now take $G = N = \mathrm U(\OO_F)$ to be our integral $p$-adic unipotent group, and let $\Lambda = \mathrm T(\OO_F)$ act on $N$ by conjugation and consequently on the group algebra $\Zp [[ N ]]$. This conjugation action clearly preserves the augmentation ideal filtration.

We want to apply our previous setup from subsection \ref{spectralsection}, and in particular corollary \ref{equivariantss}, to Symonds and Weigel's definition for the $N$-cohomology of algebraic modules $V$ (in particular, the trivial module $\Zp$) and get a spectral sequence which computes $H^*(N, V)$.
To do this, it remains to check that the $\Lambda \sharp A$-module $M = \Zp = \Zp[[N]] / I$ is filt f.g. with exhaustive filtration, but this is clear, since the $T(\OO_F)$-action by conjugation on $\Zp[[N]]$ yields the trivial action on $\Zp[[N]]/I$, and hence
specializing corollary \ref{equivariantss} to this situation we obtain the following
\equivariantss*
Notice in particular that the trivial $\Lambda \sharp \Zp[[N]]$-module $V= \Zp$ has bounded and discrete filtration, so the corollary applies to it and we obtain a spectral sequence computing $H^*(N, \Zp)$.

We have all the ingredients to prove our final goal, which we restate for the convenience of the reader.
\groupLie*
\begin{proof} First of all, by theorems \ref{unipotentLieiso} and \ref{envelopingalgebraiso} we have an isomorphism $\mathcal U(\mathfrak n) \cong \gr \Zp[[N]]$ of graded $\Zp$-Lie algebras, where the grading on the left is induced by the grading on $\mathfrak n$ given by root heights, and the grading on the right comes from the augmentation ideal filtration.

Since $N$ is a compact, $p$-adic analytic group, we have by theorem \ref{Lietogroupss} a convergent spectral sequence of graded $\mathrm T(\OO_F)$-modules \[ H^{r+s}_{\mathcal U(\mathfrak n)} ( \Zp, \Zp )_r \Rightarrow  H^{r+s} (N, \Zp)_r, \] and we want to show that it collapses on the first page.

Notice that the $E_1$-page is the Lie algebra cohomology for $\mathfrak n$.
Since by corollary \ref{kostantcor} $H^n (\mathfrak n, \Zp)$ is a finite rank, free $\Zp$-module, it suffices to show that for each $n$, $H^n (\mathfrak n, \Zp) \otimes_{\Zp} \Qp$ and $H^n( N, \Zp) \otimes_{\Zp} \Qp$ are $\Qp$-vector spaces of the same dimension, as this will force the abutment of the spectral sequence to coincide with its $E_1$-page.

We want to show that \begin{equation} \label{rationalisom} H^*_{\Zp} \left( \mathfrak n, \Zp \right) \otimes_{\Zp} \Qp \cong H^*( N, \Zp) \otimes_{\Zp} \Qp. \end{equation}

Consider first of the left hand side. 
%
Just like in lemma \ref{Liecohomologybasechange}, we can tensor with $\Qp$ the complex computing $H^*_{\Zp} \left( \mathfrak n, \Zp \right)$.
To apply the universal coefficient theorem (as in corollary 7.56 of \cite{rotman}) and conclude that $H^*_{\Zp} (\mathfrak n, \Zp) \otimes_{\Zp} \Qp \cong H^*_{\Qp}( \mathfrak n_{\Qp}, \Qp)$ it remains to check that the relevant $\Tor$ groups (that is, $\Tor_1^{\Zp} \left( H^*(\mathfrak n, \Zp), \Qp \right)$) are zero, but this is clear as $\Qp$ is a flat $\Zp$-module.

Now we switch our attention to the right side of formula \ref{rationalisom}. In section 4.2 of \cite{hkn} Huber, Kings and Neumann build on work of Lazard (\cite{lazard}) and prove, in particular, some comparison theorems between continuous group cohomology, analytic group cohomology and Lie algebra cohomology for a certain class of $p$-adic analytic groups.
We are interested in theorem 4.3.1 of loc. cit.: since $\mathrm U' / \Zp$ is a smooth group scheme with connected generic fiber, they describe an isomorphism \[ H^*_{la} (N, \Qp) \stackrel{\sim}{\lra} H^*_{\Qp}( \mathfrak n_{\Qp}, \Qp ) \] between analytic group cohomology and Lie algebra cohomology, with $\Qp$-coefficients.

The analytic group cohomology of a pro-$p$, $p$-adic analytic group has first been studied by Lazard in \cite{lazard}, chapter V. It can be thought of as the cohomology of the complex of locally analytic cochains (a subcomplex of the continuous cochains).

Lazard himself proves (theorem 2.3.10 chapter V in loc. cit.) that the embedding of the subcomplex of analytic cochains into the complex of continuous cochains yields, for any finite, torsion-free $\Zp$-module (such as $\Qp$), an isomorphism at the level of cohomology: \[ H^*_{la} (N, \Qp) \stackrel{\sim}{\lra} H^*_c (N, \Qp), \] where again we remark that the cohomology $H^*_c(N, \Qp)$ is defined as the homology of the complex of continuous cochains.

On the other hand, Symonds and Weigel in \cite{SW}, section 3.8 give a way to define the same rational continuous cohomology as an $\Ext$-group, and the two descriptions coincide (as is clear from Symonds and Weigel' construction of rational cohomology, where they use any projective resolution of $\Zp$ by profinite $N$-modules).

Finally, Symonds and Weigel prove in loc. cit. theorem 3.8.2 that \[ H^*_c(N, \Zp) \otimes_{\Zp} \Qp \cong H^*_c(N, \Qp), \] which concludes our proof.
\end{proof}

\end{document}